\documentclass[11pt]{amsart}
\usepackage{amscd,amsfonts,amssymb,amsmath}
\usepackage{hyperref}
\usepackage{epsfig}
\usepackage{mathtools}
\usepackage{tabu}
\usepackage{tikz-cd}
\usepackage{doi}
\usepackage{siunitx}
\newtheorem{theorem}{Theorem}[section]
\newtheorem{corollary}[theorem]{Corollary}
\newtheorem{lemma}[theorem]{Lemma}
\newtheorem{proposition}[theorem]{Proposition}

\newtheorem{example}[theorem]{Example}
\newtheorem{remark}[theorem]{Remark}
\numberwithin{equation}{subsection}
\newtheorem*{ack}{Acknowledgement}
\usepackage[all,cmtip]{xy}

\usepackage{graphicx}

\newcommand{\Aut}{\operatorname{Aut}}

\newcommand{\id}{\mathrm{id}}

\newcommand{\restr}[1]{|_{#1}}
\begin{document}

\title{Alexander and Markov theorems for virtual doodles}

\author{Neha Nanda}
\author{Mahender Singh}
\address{Department of Mathematical Sciences, Indian Institute of Science Education and Research (IISER) Mohali, Sector 81,  S. A. S. Nagar, P. O. Manauli, Punjab 140306, India.}
\email{nehananda94@gmail.com}
\email{mahender@iisermohali.ac.in}


\begin{abstract}
Study of certain isotopy classes of a finite collection of immersed circles without triple or higher intersections on closed oriented surfaces can be thought of as a planar analogue of virtual knot theory where the genus zero case corresponds to classical knot theory. Alexander and Markov theorems for the genus zero case are known where the role of groups is played by twin groups, a class of right angled Coxeter groups with only far commutativity relations. The purpose of this paper is to prove Alexander and Markov theorems for higher genus case where the role of groups is played by a new class of groups called virtual twin groups which extends twin groups in a natural way.
\end{abstract}

\maketitle

\section{Introduction}
The study of doodles on surfaces began with the work of Fenn and Taylor \cite{FennTaylor} who defined a doodle as a finite collection of simple closed curves lying in a 2-sphere without triple or higher intersections. The idea was extended by Khovanov \cite{Khovanov} to a finite collection of closed curves without triple or higher intersections on a closed oriented surface. An analogue of the link group for doodles was also introduced in \cite{Khovanov} and several infinite families of doodles whose fundamental groups have infinite centre were constructed. Recently, Bartholomew-Fenn-Kamada-Kamada \cite{BartholomewFennKamada2018} extended the study of doodles to immersed circles on a closed oriented surface of any genus, which can be thought of as virtual links analogue for doodles. An invariant of virtual doodles by coloring their diagrams using a special type of algebra has been constructed in \cite{BartholomewFennKamada2019}. Recently, an Alexander type invariant for oriented doodles which vanishes on unlinked doodles with more than one component has been constructed in \cite{CisnerosFloresJuyumayaMarquez}. 
\par

The role of groups for doodles on a  2-sphere is played by  twin groups. The twin groups $T_n$, $n \ge 2$, form a special class of right angled Coxeter groups and appeared in the work of Shabat and Voevodsky \cite{ShabatVoevodsky}, who referred them as Grothendieck cartographical groups. Later, these groups were investigated by Khovanov \cite{Khovanov} under the name twin groups, who also gave a geometric interpretation of these groups similar to the one for classical braid groups. Consider configurations of $n$ arcs in the infinite strip $\mathbb{R} \times  [0,1]$ connecting $n$ marked points on each of the parallel lines $\mathbb{R} \times \{1\}$ and $\mathbb{R} \times \{0\}$ such that each arc is monotonic and no three arcs have a point in common. Two such configurations are equivalent if one can be deformed into the other by a homotopy of such configurations in $\mathbb{R} \times [0,1]$ keeping the end points of arcs fixed. An equivalence class under this equivalence is called a twin. The product of two twins can be defined by placing one twin on top of the other, similar to that in the braid group $B_n$. The collection of all twins with $n$ arcs under this operation forms a group isomorphic to $T_n$. Taking the one point compactification of the plane, one can define the closure of a twin on a $2$-sphere analogous to the closure of a braid in $\mathbb{R}^3$. Khovanov also  proved an analogue of the classical Alexander Theorem for doodles on a 2-sphere, that is, every oriented doodle on a $2$-sphere is closure of a twin. An analogue of the Markov Theorem for doodles on a 2-sphere has been established recently by Gotin \cite{Gotin}. From a wider perspective, a recent work \cite{BartholomewFenn2019} look at which Alexander and Markov theories can be defined for generalized knot theories.
\par

The pure twin group $PT_n$ is the kernel of the natural surjection from $T_n$ onto the symmetric group $S_n$ on $n$ symbols. Algebraic study of twin and pure twin groups has recently attracted a lot of attention.   In a recent paper \cite{BarVesSin}, Bardakov-Singh-Vesnin proved that $PT_n$ is free for $n = 3,4$ and not free for $n \geq 6$. Gonz\'alez-Le\'on-Medina-Roque \cite{GonGutiRoq} recently showed that $PT_5$ is a free group of rank $31$.  A lower bound for the number of generators of $PT_n$ is given in \cite{HarshmanKnapp} while an upper bound is given in \cite{BarVesSin}. It is worth noting that \cite{HarshmanKnapp} physicists refer twin and pure twin groups as  traid and pure traid groups, respectively.  Description of $PT_6$ has been obtained recently by Mostovoy and Roque-M\'arquez \cite{MostRoq} where they prove that $PT_6$ is a free product of the free group $F_{71}$ and 20 copies of the free abelian group $\mathbb{Z} \oplus \mathbb{Z}$. A complete presentation of $PT_n$ for $n \ge 7$ is still not known and seems challenging to describe. Automorphisms, conjugacy classes and centralisers of involutions in twin groups have been explored in recent works \cite{NaikNandaSingh1, NaikNandaSingh2}. 
\par

One can think of the study of isotopy classes of immersed circles without triple or higher intersection points on closed oriented surfaces as a planar analogue of virtual knot theory with the genus zero case corresponding to classical knot theory. As mentioned earlier Alexander and Markov theorems for the genus zero case are already known in the literature where the role of groups is played by twin groups. The purpose of this paper is to prove Alexander and Markov theorems for higher genus case. We show that virtual twin groups introduced in a recent work \cite{BarVesSin} as abstract generalisation of twin groups play the role of groups for the theory of virtual doodles. A virtual twin group extends a twin group and surjects onto a symmetric group in a natural way. A pure analogue of the virtual twin group is defined analogously as the kernel of the natural surjection onto the symmetric group.
\par

The paper is organised as follows. We define twin and virtual twin groups in Section \ref{basics}. A topological interpretation of virtual twins is given in Section \ref{topo-interpret-virtual-twins}. We discuss virtual doodle diagrams and their Gauss data in Section \ref{virtual-doodle-diagram}. Finally, we prove Alexander Theorem for virtual doodles in Section \ref{Alexander-theorem-section} and Markov Theorem in Section \ref{Markov-theorem}.

\section{Twin and virtual twin groups}\label{basics}
For an integer $n \ge 2$, the \textit{twin group} $T_n$ is defined as the group with the presentation
$$\big\langle s_1, s_2, \dots, s_{n-1}~|~ s_i^{2} = 1~ \text{for}~1 \le i \le n-1~\textrm{and}~ s_is_j = s_js_i~ \text{for}~ |i-j| \geq 2 \big\rangle.$$
Elements of $T_n$ are called twins and the generator $s_i$ can be geometrically presented by a configuration shown in Figure \ref{Twin}.

\begin{figure}[hbtp]
\centering
\includegraphics[scale=0.45]{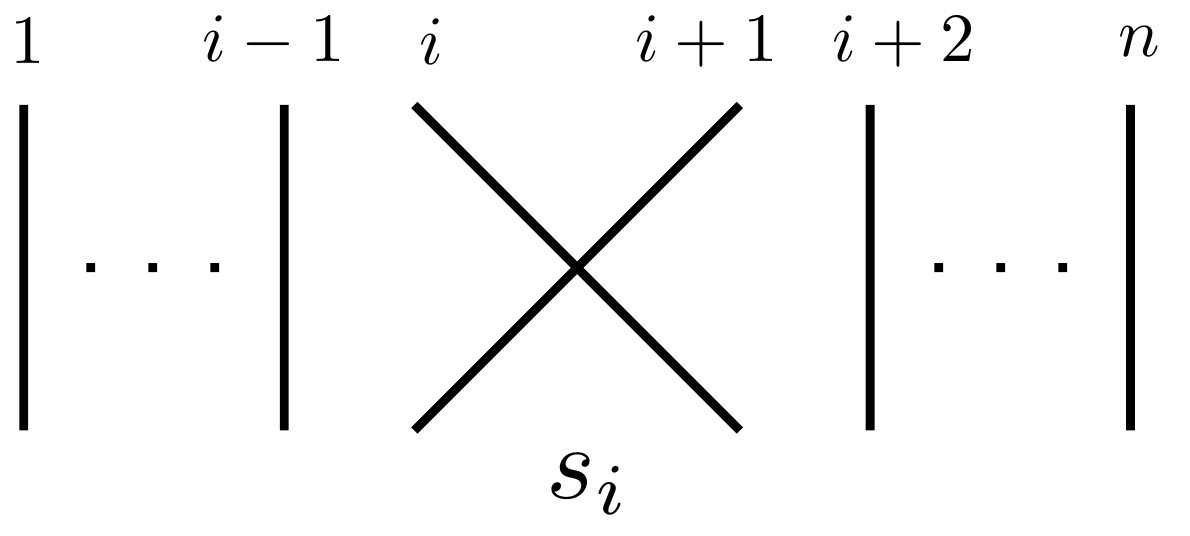}
\caption{The twin $s_i$}
\label{Twin}
\end{figure}

The kernel of the natural surjection from $T_n$ onto $S_n$, the symmetric group on $n$ symbols, is called the \textit{pure twin group} and is denoted by $PT_n$.

The {\it virtual twin group} $VT_n$, $n \ge 2$, was introduced in \cite[Section 5]{BarVesSin} as an abstract generalisation of the twin group $T_n$. The abstract group $VT_n$ has generators $\{ s_1, s_2, \ldots, s_{n-1}, \rho_1, \rho_2, \ldots, \rho_{n-1}\}$ and defining relations
\begin{eqnarray}\label{virtual-twin-relations}
s_i^{2} &=&1 ~~ \textrm{for } i = 1, 2, \dots, n-1,\\
\nonumber s_is_j &=& s_js_i ~~ \textrm{for } |i - j| \geq 2,\\
\nonumber \rho_i^{2} &=& 1 ~~ \textrm{for } i = 1, 2, \dots, n-1,\\
\nonumber \rho_i\rho_j &=& \rho_j\rho_i ~~ \textrm{for } |i - j| \geq 2,\\
\nonumber \rho_i\rho_{i+1}\rho_i &=& \rho_{i+1}\rho_i\rho_{i+1}~ ~\textrm{for } i = 1, 2, \dots, n-2,\\
\nonumber \rho_is_j &=& s_j\rho_i ~~ \textrm{for } |i - j| \geq 2,\\
\nonumber \rho_i\rho_{i+1} s_i &=& s_{i+1} \rho_i \rho_{i+1}~ ~\textrm{for } i = 1, 2, \dots, n-2.
\end{eqnarray}
The kernel of the natural surjection from $VT_n$ onto $S_n$ is called the \textit{virtual pure twin group} and is denoted by $VPT_n$. We show that virtual twin groups play the role of groups in the theory of virtual doodles.

\section{Topological interpretation of virtual twins}\label{topo-interpret-virtual-twins}
Consider a set $Q_n$ of $n$ points in $\mathbb{R}$. A \textit{virtual twin diagram} on $n$ strands is a subset $D$ of $\mathbb{R} \times [0,1]$ consisting of $n$ intervals called {\it strands} with $\partial D = Q_n  \times \{0,1\}$ and satisfying the following conditions:

\begin{enumerate}
\item the natural projection $\mathbb{R} \times [0,1] \to [0,1]$ maps each strand homeomorphically onto the unit interval $[0,1]$, 
\item the set $V(D)$ of all crossings of the diagram $D$ consists of transverse double points of $D$ where each crossing has the pre-assigned information of being a real or a virtual crossing as depicted in Figure \ref{Crossings}. A virtual crossing is depicted by a crossing encircled with a small circle.

\end{enumerate}
\begin{figure}[hbtp]
\centering
\includegraphics[scale=0.4]{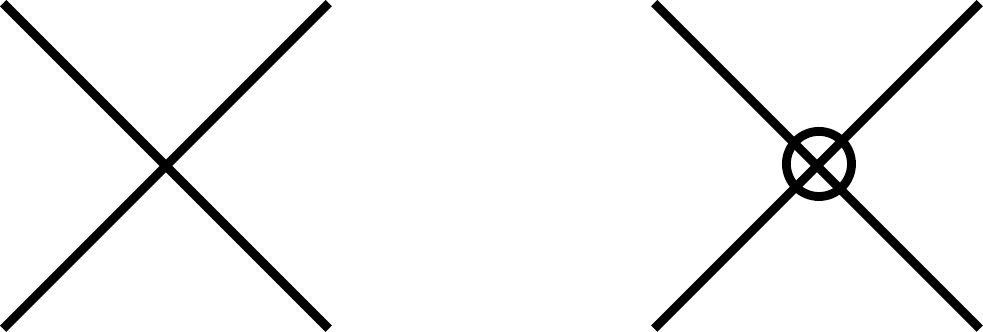}
\caption{Real and virtual crossings}
\label{Crossings}
\end{figure}

Two virtual twin diagrams $D_1$ and $D_2$ on $n$ strands are said to be \textit{equivalent} if one can be obtained from the other by a finite sequence of moves as shown in Figure \ref{ReidemeisterMoves} and isotopies of the plane. We define a \textit{virtual twin} as an equivalence class of such virtual twin diagrams. Let $\mathcal{VT}_n$ denote the set of all virtual twins on $n$ strands. The product $D_1D_2$ of two virtual twin diagrams $D_1$ and $D_2$ is defined by placing $D_1$ on top of $D_2$ and then shrinking the interval to $[0,1]$. It is clear that if  $D_1$ is equivalent to $D_1'$ and $D_2$ is equivalent to $D_2'$ , then $D_1D_2$ is equivalent to $D_1'D_2'$. Thus, there is a well-defined binary operation on the set $\mathcal{VT}_n$. It is easy to observe that this operation is indeed associative.

\begin{figure}[hbtp]
\centering
\includegraphics[scale=0.3]{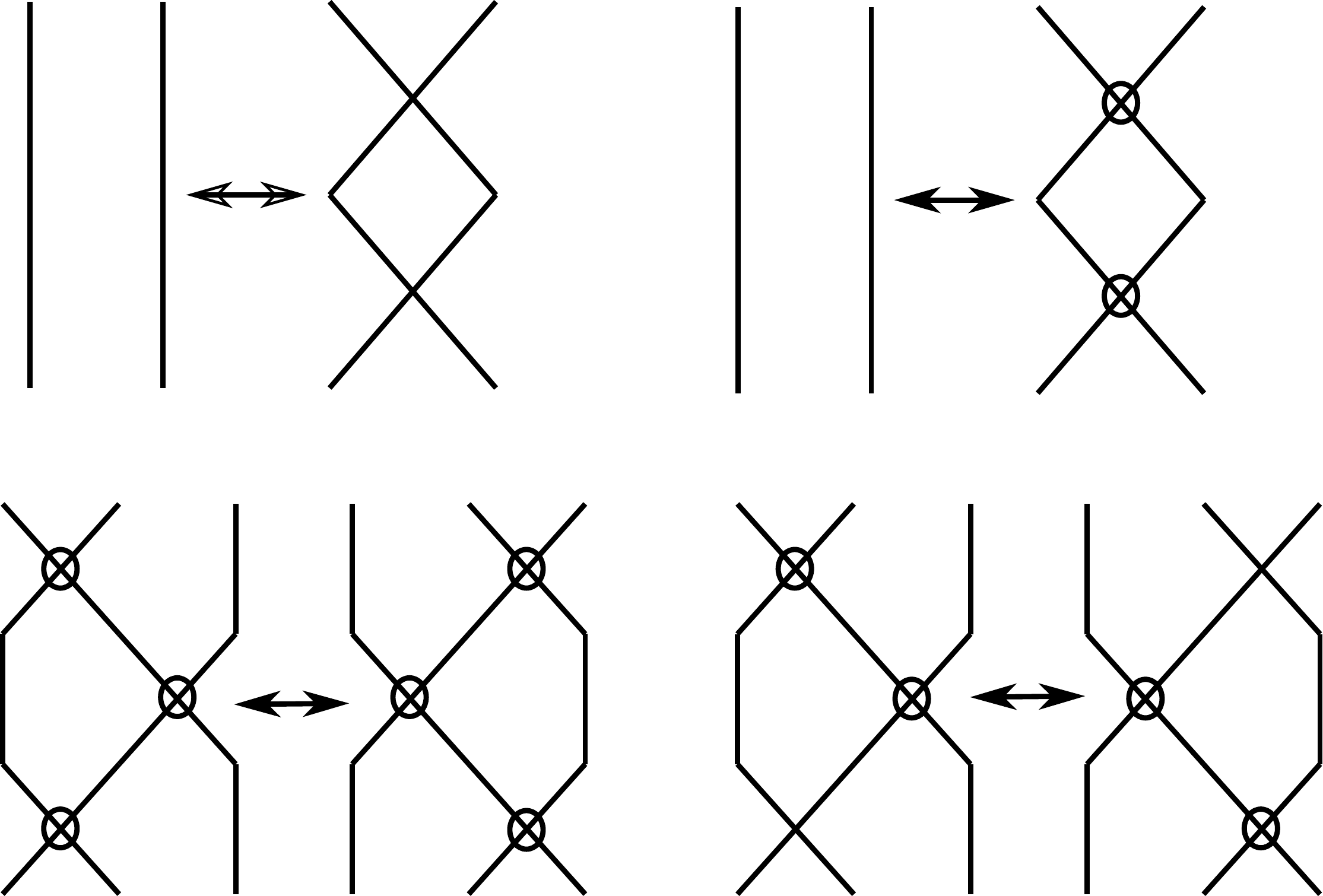}
\caption{Moves for virtual twin diagrams}
\label{ReidemeisterMoves}
\end{figure}

\begin{figure}[hbtp]
\centering
\includegraphics[scale=0.27]{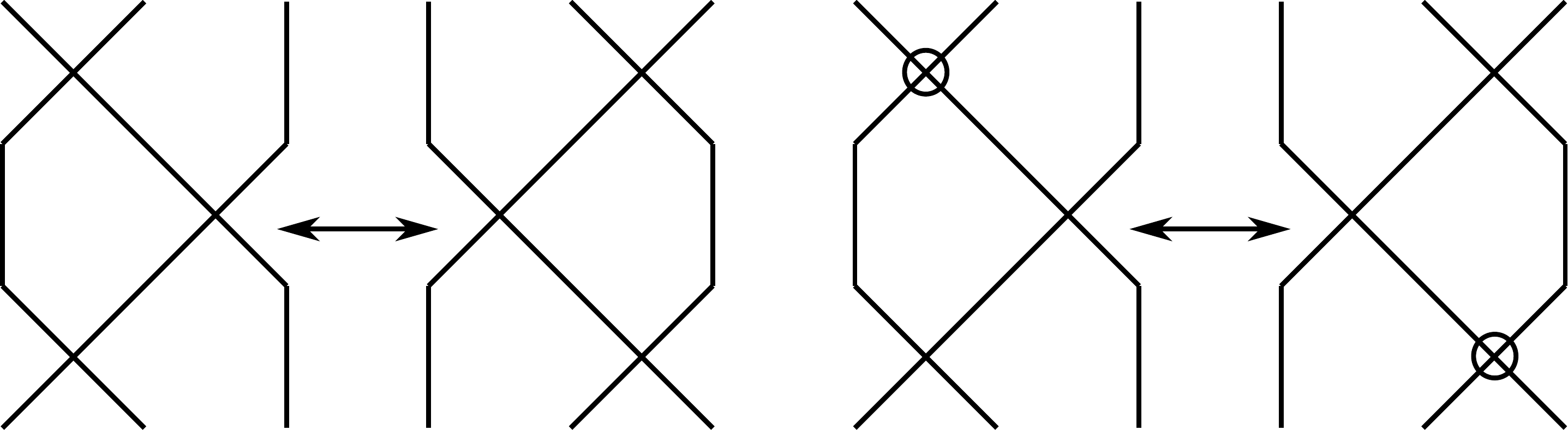}
\caption{Forbidden Moves}
\label{ForbiddenMoves}
\end{figure}

\begin{remark}
 Every classical link diagram can be regarded as an immersion of circles in the plane with an extra structure (of over/under crossing) at the double points. If we take a diagram without this extra structure, then it is simply a shadow of some link in $\mathbb{R}^3$ and such crossings are called {\it flat crossings} in the literature \cite{Kauffman1999}.  An easy check shows that if one is allowed to apply the classical Reidemeister moves to such a diagram, then the diagram can be reduced to a disjoint union of circles. However, this does not happen in flat virtual diagrams, that is, diagrams which have both flat and virtual crossings. It is worth noting that if we include the first forbidden move in the moves for virtual twin diagrams, then we get precisely the theory of {\it flat virtual links} initiated in \cite{Kauffman1999}. We note that the moves in Figure \ref{ForbiddenMoves} are forbidden and cannot be obtained from moves in Figure \ref{ReidemeisterMoves} (see Proposition \ref{Justification For Forbidden Moves}). 
 \end{remark}

\begin{figure}[hbtp]
\centering
\includegraphics[scale=0.45]{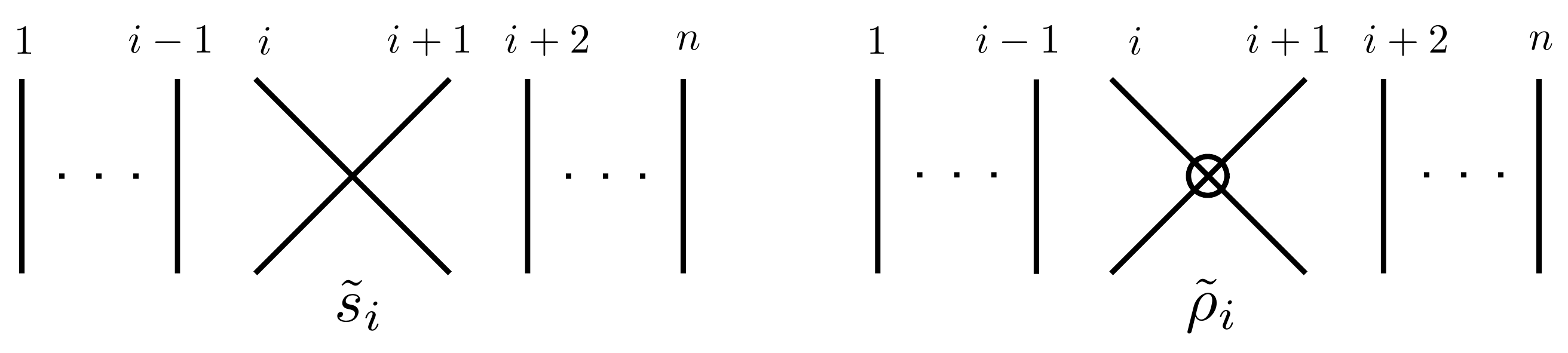}
\caption{Generators $\tilde{s}_i$ and $\tilde{\rho}_i$}
\label{Generators}

\end{figure}
\begin{lemma} For each $n \geq 2$, the set $\mathcal{VT}_n$ of virtual twins forms a group under the operation defined above.
\end{lemma}

\begin{proof}
We begin by noting that the virtual twin represented by a diagram of $n$ strands with no crossings is the identity element of the set $\mathcal{VT}_n$ of virtual twins. For each $i=1, 2, \dots , n-1$, let us define $\tilde s_i$ and $\tilde \rho_i$ to be the virtual twins represented by diagrams as in Figure \ref{Generators}. Let $\beta$ be any arbitrary element in $\mathcal{VT}_n$. Then after applying isotopies of the plane $\beta$ can be represented by a diagram $D \subset \mathbb{R} \times [0,1]$ such that the projection $\mathbb{R} \times [0,1] \to [0,1]$ restricted to the set $V(D)$ of all crossings is injective,  that is, each crossing is at a distinct level.  Further, it follows from the moves given in Figure \ref{ReidemeisterMoves} that ${\tilde s_i}^{2}=1$ and $\tilde \rho_i^{2}=1$ for all $i= 1, 2, \dots , n-1$. Thus, we can write $\beta={\tilde{s}_{i_1}}^{\epsilon_1}\tilde{\rho}_{i_2}^{\epsilon_2} \dots \tilde{s}_{i_k}^{\epsilon_k}$ for some $k$, where $\epsilon_i \in \{0,1\}$. Since $\tilde{s}_i$ and $\tilde{\rho}_i$ are self inverses, the element $\beta$ has the inverse $\tilde{s}_{i_k}^{\epsilon_k} \dots \tilde{\rho}_{i_2}^{\epsilon_2}{\tilde{s}_{i_1}}^{\epsilon_1}$. 
\end{proof}

\begin{proposition}
The diagrammatic group $\mathcal{VT}_n$ and the abstract group $VT_n$ are isomorphic for all $n \geq 2$.
\end{proposition}
\begin{proof}
It follows from the definition of equivalence of two virtual twin diagrams on $n$ strands that the generators $\tilde s_i$ and $\tilde{\rho_i}$ satisfy the following relations.
\begin{eqnarray*}
\tilde{s}_i^{2} &=&1 ~~ \textrm{for } i = 1, 2, \dots, n-1,\\
\tilde{s}_i\tilde{s}_j &=& \tilde{s}_j\tilde{s}_i ~~ \textrm{for } |i - j| \geq 2,\\
\tilde{\rho}_i^{2} &=& 1 ~~ \textrm{for } i = 1, 2, \dots, n-1,\\
\tilde{\rho}_i\tilde{\rho}_j &=& \tilde{\rho}_j\tilde{\rho}_i ~~ \textrm{for } |i - j| \geq 2,\\
\tilde{\rho}_i\tilde{\rho}_{i+1}\tilde{\rho}_i &=& \tilde{\rho}_{i+1}\tilde{\rho}_i\tilde{\rho}_{i+1}~ ~\textrm{for } i = 1, 2, \dots, n-2,\\
\tilde{\rho}_i\tilde{s}_j &=& \tilde{s}_j\tilde{\rho}_i ~~ \textrm{for } |i - j| \geq 2,\\
\tilde{\rho}_i\tilde{\rho}_{i+1} \tilde{s}_i &=& \tilde{s}_{i+1} \tilde{\rho}_i \tilde{\rho}_{i+1}~ ~\textrm{for } i = 1, 2, \dots, n-2.
\end{eqnarray*}
Thus, there exists a unique group homomorphism $$f_n : VT_n \to \mathcal{VT}_n$$ given by $f_n(s_i)=\tilde{s_i}$ and $f_n(\rho_i)=\tilde{\rho_i}$ for $i= 1, 2, \dots, n-1$. Since every $\beta \in \mathcal{VT}_n$ can be written as a product of $\tilde s_i$ and $\tilde{\rho_i}$, the map $f_n$ is surjective. For an element $\tilde{s}_{i_1}^{\epsilon_1}\tilde{\rho}_{i_2}^{\epsilon_2} \dots \tilde{s}_{i_k}^{\epsilon_k} \in \mathcal{VT}_n$, where $\epsilon_i \in \{0,1\}$, define $$g_n : \mathcal{VT}_n \to VT_n$$ by $g_n\big(\tilde{s}_{i_1}^{\epsilon_1}\tilde{\rho}_{i_2}^{\epsilon_2} \dots \tilde{s}_{i_k}^{\epsilon_k} \big)=s_{i_1}^{\epsilon_1}\rho_{i_2}^{\epsilon_2} \dots s_{i_k}^{\epsilon_k}$. We prove that $g_n$ is well-defined. Let $D$ be a virtual twin diagram representing the element $\tilde{s}_{i_1}^{\epsilon_1}\tilde{\rho}_{i_2}^{\epsilon_2} \dots \tilde{s}_{i_k}^{\epsilon_k}$.  A diagram obtained by a planar isotopy on $D$ that does not change the order of the image of $V(D)$ in $[0, 1]$ under the projection map  $\mathbb{R} \times [0,1] \to [0,1]$ is again represented by the element $\tilde{s}_{i_1}^{\epsilon_1}\tilde{\rho}_{i_2}^{\epsilon_2} \dots \tilde{s}_{i_k}^{\epsilon_k}$.  Any move that interchanges two points in the image of $V(D)$ under the projection $\mathbb{R} \times [0,1] \to [0,1]$ exchanges the subwords $\tilde{s}_i\tilde{s_j}$ and  $\tilde{s}_j\tilde{s_i}$, $\tilde{s}_i\tilde{\rho}_j$ and $\tilde{\rho}_j\tilde{s}_i$ or $\tilde{\rho}_i\tilde{\rho}_j$ and $\tilde{\rho}_j\tilde{\rho}_i$ in the word $\tilde{s}_{i_1}^{\epsilon_1}\tilde{\rho}_{i_2}^{\epsilon_2} \dots \tilde{s}_{i_k}^{\epsilon_k}$ for some $|i-j| \geq 2$. Under each of these cases, the images of the  corresponding words under $g_n$ are the same element in $VT_n$. The move that adds (respectively, removes) two points in $V(D)$ adds (respectively, removes) subwords of the form $\tilde{s}_i\tilde{s_i}$  or  $\tilde{\rho}_i\tilde{\rho}_i$ in the word $\tilde{s}_{i_1}^{\epsilon_1}\tilde{\rho}_{i_2}^{\epsilon_2} \dots \tilde{s}_{i_k}^{\epsilon_k}$.  But $s_i^2=1=\rho_i^2$ in $VT_n$, and hence both the words are mapped to same element under $g_n$. The third move interchanges the subwords $\tilde{\rho}_i\tilde{\rho}_{i+1}\tilde{\rho}_i$ and $\tilde{\rho}_{i+1}\tilde{\rho}_i\tilde{\rho}_{i+1}$ in the word $\tilde{s}_{i_1}^{\epsilon_1}\tilde{\rho}_{i_2}^{\epsilon_2} \dots \tilde{s}_{i_k}^{\epsilon_k}$. But $VT_n$ has the relation $\rho_i\rho_{i+1}\rho_i = \rho_{i+1}\rho_i\rho_{i+1}$. Finally, the last move replaces the subwords $\tilde{\rho}_i\tilde{\rho}_{i+1}\tilde{s}_i$ and $\tilde{s}_{i+1}\tilde{\rho}_i\tilde{\rho}_{i+1}$, but $VT_n$ has the relation $\rho_i\rho_{i+1} s_i = s_{i+1} \rho_i \rho_{i+1}$, and hence $g_n$ is well-defined. Since $g_n \circ f_n = \id$, $f_n$ is injective and the proof is complete.
 \end{proof}
 
Since the diagrammatic group $\mathcal{VT}_n$ and the abstract group $VT_n$ have been identified, from now onwards, the generators $s_i$ and $\rho_i$ will be represented geometrically as in Figure \ref{Generators}.

A representation $\mu_n: T_n \to \Aut(F_n)$, from the twin group to the automorphisms group of the free group, has been constructed in \cite[Theorem 7.1]{NaikNandaSingh1}. It turns out that $\mu_n$ extends easily to a representation of $VT_n$. 

\begin{proposition}\label{VTn-representation}
The map $\mu_n: VT_n \to \Aut(F_n)$ defined by the action of generators of $VT_n$ by
 \[ \mu_n(s_i) :
 \begin{cases}
 x_i \mapsto x_ix_{i+1},\\
 x_{i+1} \mapsto x_{i+1}^{-1},\\
 x_j \mapsto x_j, \quad j \neq i, i+1,\\
 \end{cases}
 \]
 \[ \mu_n(\rho_i) :
 \begin{cases}
 x_i \mapsto x_{i+1},\\
 x_{i+1} \mapsto x_i\\
 x_j \mapsto x_j, \quad j \neq i, i+1,\\
 \end{cases}
 \]
 is a representation of $VT_n$.
\end{proposition}

As a consequence of Proposition \ref{VTn-representation}, it follows that the forbidden moves in Figure \ref{ForbiddenMoves} cannot be obtained from the moves in Figure \ref{ReidemeisterMoves}.

\begin{proposition}\label{Justification For Forbidden Moves}
The following holds in $VT_n$:
\begin{enumerate}
\item $s_i s_{i+1} s_{i} \neq s_{i+1} s_i s_{i+1}.$
\item $\rho_i s_{i+1} s_{i} \neq s_{i+1} s_i \rho_{i+1}.$
\end{enumerate}
\end{proposition}

\begin{proof}
An easy check gives $$\mu_n(s_i s_{i+1} s_{i})(x_i) \neq  \mu_n(s_{i+1} s_{i} s_{i+1})(x_i)$$ and 
$$\mu_n(\rho_i s_{i+1} s_{i})(x_i) \neq  \mu_n(s_{i+1} s_{i} \rho_{i+1})(x_i)$$ for each $i$.
\end{proof}

\section{Virtual doodle diagrams}\label{virtual-doodle-diagram}
 A \textit{virtual doodle diagram} is a generic immersion of a closed one-dimensional manifold (disjoint union of circles) on the plane $\mathbb{R}^2$ with finitely many real or virtual crossings (as in Figure \ref{Crossings}) such that there are no triple or higher real intersection points.
 
 \begin{example}
 An example of a virtual doodle is shown in Figure \ref{KishinoDoodle}. The figure  represents a flat virtual knot called the flat Kishino knot which was proved to be non-trivial as a flat virtual knot in \cite{FennTuraev2007, Kadokami2003}.  Thus, the flat Kishino knot is also non-trivial as a virtual doodle. Note that, the original Kishino knot diagram is a diagram of a virtual knot and its non-triviality as a virtual knot   is proven, for example, in
\cite{BartholomewFenn2008, KishinoSatoh2004}.

 \begin{figure}[hbtp]
 \centering
 \includegraphics[scale=0.4]{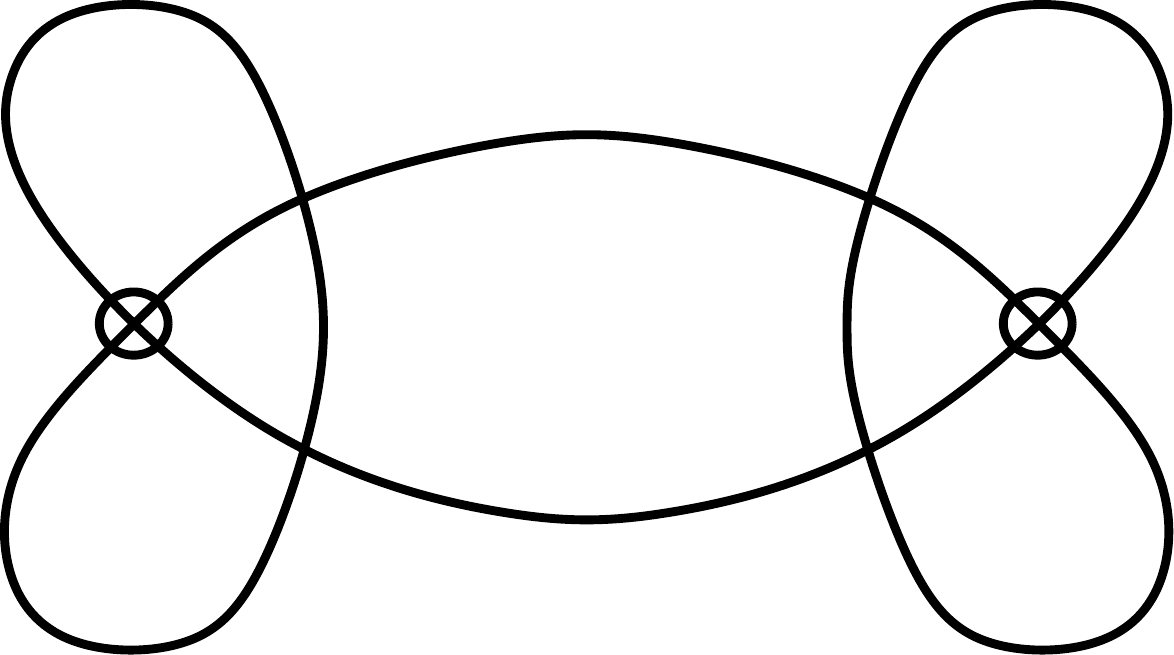}
 \caption{Flat Kishino knot as virtual doodle}
 \label{KishinoDoodle}
 \end{figure}
\end{example}
 
Two virtual doodle diagrams are \textit{equivalent} if they are related by a finite sequence of $R_1$, $R_2$, $VR_1$, $VR_2$, $VR_3$, $M$ moves as shown in Figure \ref{ReidemeisterMovesForDoodles} and isotopies of the plane. Note that $VR_1$, $VR_2$, $VR_3$ and $M$ are flat versions of virtual Reidemeister moves in virtual knot theory \cite{Kauffman1999}. The moves $R_1$ and $R_2$ are referred as flat versions of Reidemeister moves for classical knots \cite{BartholomewFennKamada2018-2}.
 
 \begin{figure}[hbtp]
 \centering
 \includegraphics[scale=0.3]{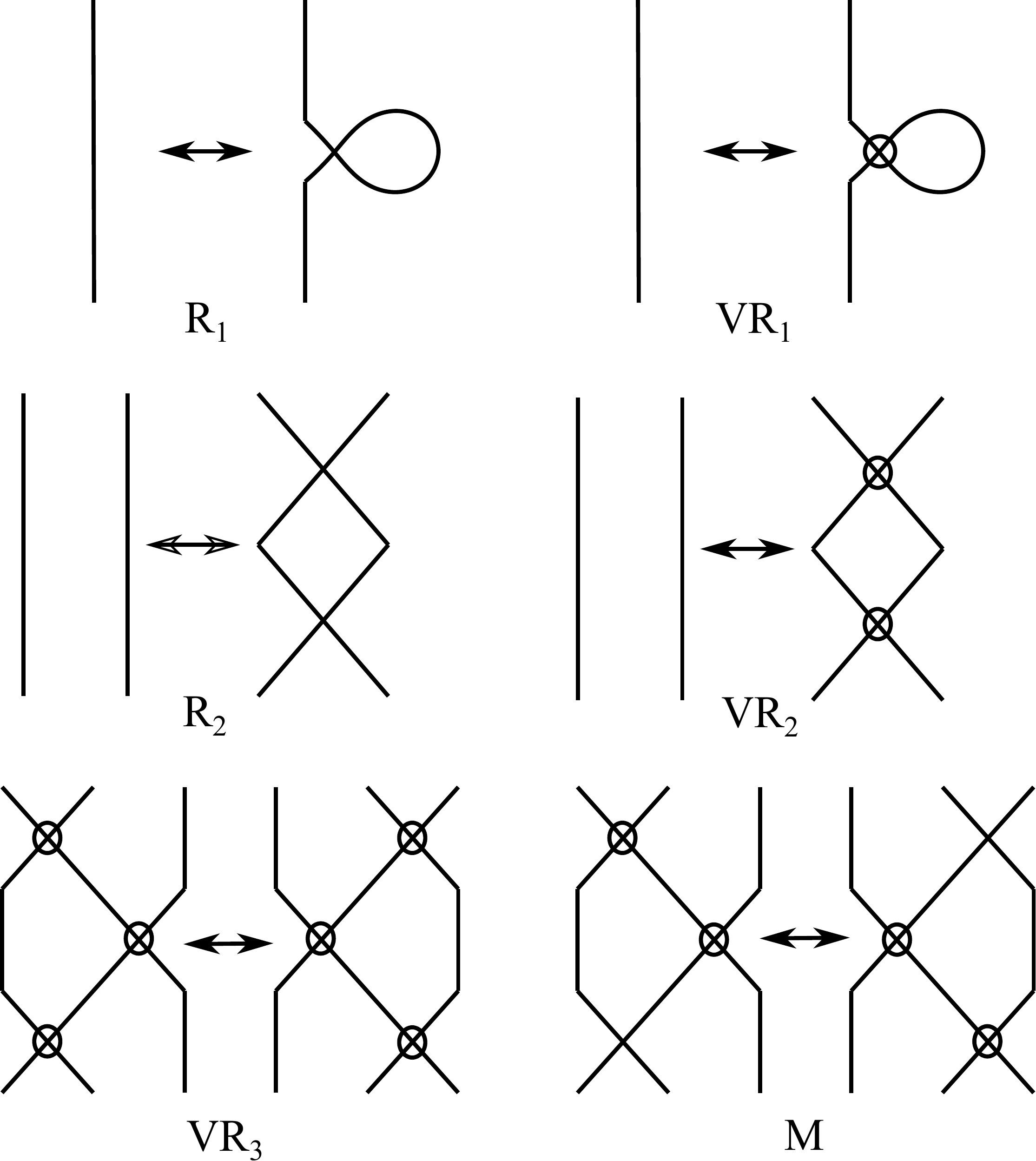}
\caption{Moves for virtual doodle diagrams}
\label{ReidemeisterMovesForDoodles}
 \end{figure}

An {\it oriented virtual doodle diagram} is a doodle diagram with an orientation on each component of the underlying immersion. It is easy to see that there are a total of  $28$ moves for oriented  virtual doodle diagrams. Further, any oriented move can be obtained as a composition of moves in Figure \ref{OrientedReidemeisterMoves} and planar isotopies. From here onwards, by a virtual doodle diagram we mean an oriented virtual doodle diagram unless stated otherwise.
 
 \begin{figure}[hbtp]
 \centering
 \includegraphics[scale=0.15]{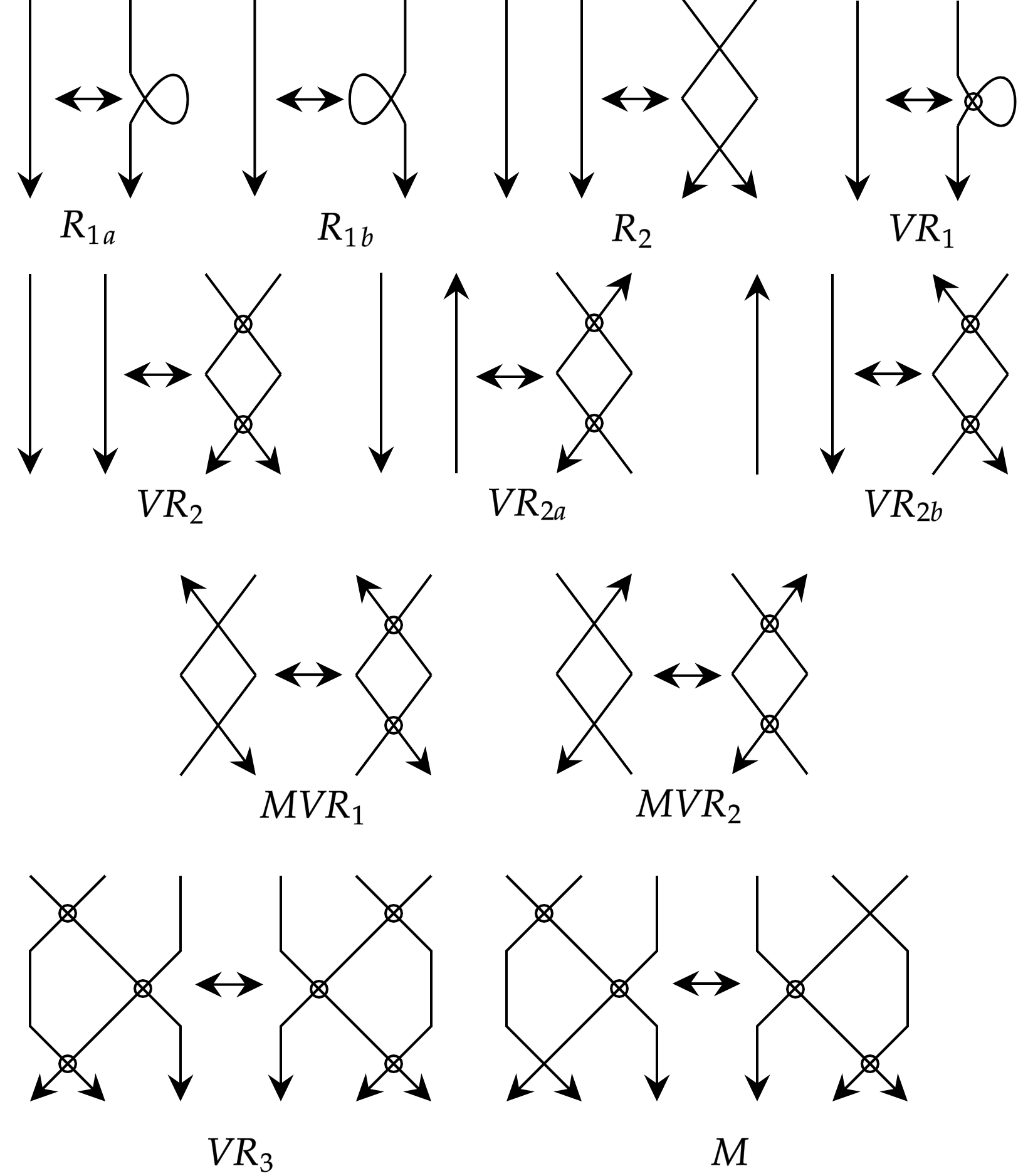}
 \caption{Moves for oriented virtual doodle diagrams}
 \label{OrientedReidemeisterMoves}
 \end{figure}
 
It is known due to \cite{BartholomewFennKamada2018} that there is a natural bijection between the set of oriented (or unoriented) virtual doodles on the plane and the set of oriented (or unoriented) doodles on surfaces. This is an analogue of a similar fact that there is a natural bijection between the set of oriented (or unoriented) virtual knots and the set of stable equivalent classes of oriented (or unoriented) knot diagrams on surfaces \cite{CarterKamadaSaito2002, KamadaKamada2000, Kuperberg2003}.
 
{ \bf Gauss data.} Let $K$ be a virtual doodle diagram on the plane with $n$ real crossings. Let $N_1, N_2, \ldots, N_n$ be closed $2$-disks each  enclosing exactly one real crossing of the diagram $K$ and $W(K)$ the closure of $\mathbb{R}^2 \setminus \cup_{i=1}^n N_i$ in the plane. Note that $W(K)$ consists of immersed arcs and loops in the plane where the intersection points are the virtual crossings.  Let $V_R(K)$ be the set of real crossings of $K$. Since we are considering oriented virtual doodle diagrams, for each real crossing $c_i$, the set $\partial N_i \cap c_i$ consists of four points and are assigned symbols as in Figure \ref{LabellingsOnRealCrossing}. 

\begin{figure}[hbtp]
\centering
\includegraphics[scale=0.36]{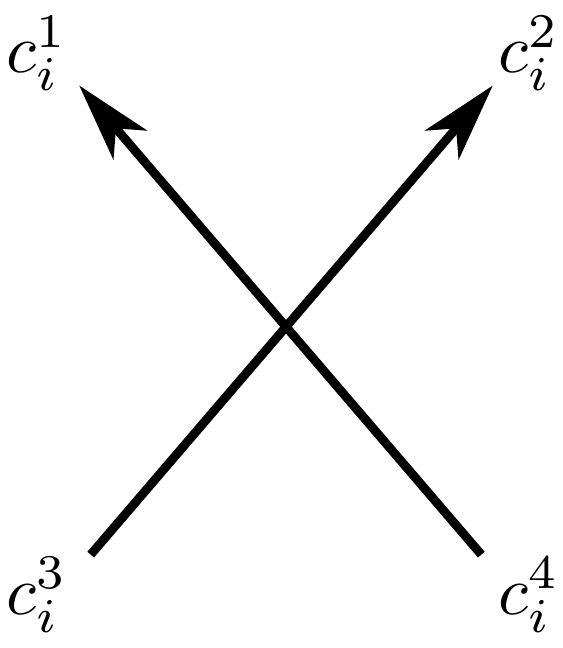}
\caption{Labelling at real crossing}
\label{LabellingsOnRealCrossing}
\end{figure}

Define $$V_\partial (K) = \big\{ c^j _i ~|~ i = 1, 2, \dots, n~ \textrm{and } j = 1, 2, 3, 4 \big\}$$ and 
\begin{eqnarray*}
X(K) &=& \big\{(a,b) \in V_\partial (K) \times V_\partial (K)~|~\textrm{there is an arc in $K \cap W(K)$ starting}\\
&&\textrm{at $a$ and ending at $b$} \big\}.
\end{eqnarray*}
We define the \textit{Gauss data} of a virtual doodle diagram $K$ to be the pair $\big(V_R(K), X(K)\big)$. See \cite[Section 6]{BartholomewFennKamada2018} for a related discussion. The Gauss data will be crucial in establishing Alexander and Markov theorems for virtual doodles which we prove in the remaining two sections. 

Let $K$ and $K'$ be two virtual doodle diagrams each with $n$ real crossings. We say that $K$ and $K'$ have the \textit{same Gauss data} if there is a bijection $\sigma: V_R(K) \to V_R(K')$ such that whenever $(a,b) \in X(K)$, then $\big(\bar{\sigma}(a), \bar{\sigma}(b)\big) \in X(K')$, where $\bar{\sigma}: V_\partial (K) \to V_\partial (K')$ is defined as 
$$\bar{\sigma}(c^j_i) = {\sigma(c_i)}^j.$$

The following result is proved in \cite[Lemma 6.1]{BartholomewFennKamada2018}. 

\begin{lemma}\label{MainLemma}
Let $K$ and $K'$ be virtual doodle diagrams with the same number of real crossings. Then the following are equivalent:
\begin{enumerate}
\item $K$ and $K'$ have the same Gauss data,
\item $K$ and $K'$ are related by a finite sequence of $VR_1$, $VR_2$, $VR_3$, $M$ moves and isotopies of the plane,
\item $K$ and $K'$ are related by a finite sequence of Kauffman's detour moves (shown in Figure \ref{DetourMove}) and isotopies of the plane.

\end{enumerate}
\end{lemma}

\begin{figure}[hbtp]
\centering
\includegraphics[scale=0.15]{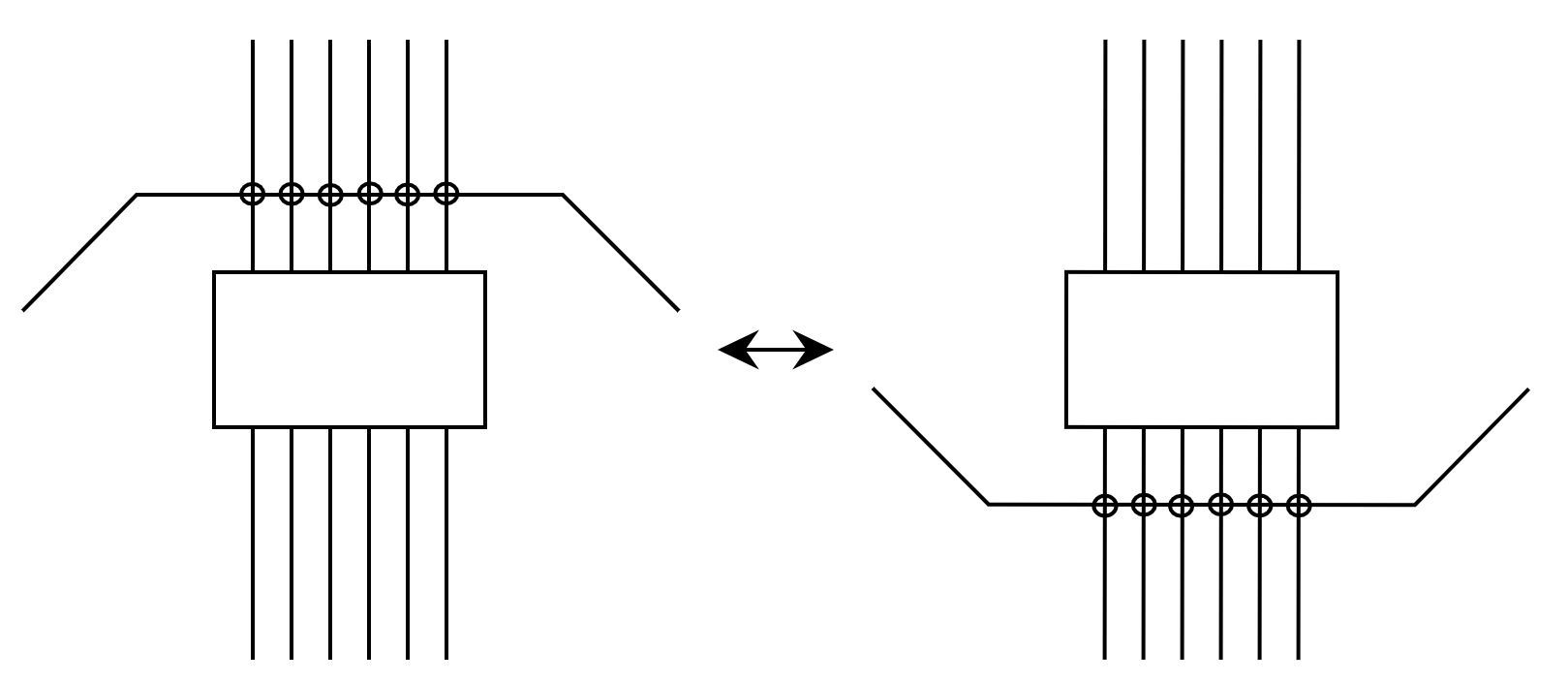}
\caption{Kauffman's detour move}
\label{DetourMove}
\end{figure}

\section{Alexander Theorem for virtual doodles}\label{Alexander-theorem-section}
Consider the space $\mathbb{R}^2 \setminus \mathbb{D}^{\circ}$, where $\mathbb{D}^{\circ}$ is the interior of the closed unit 2-disk $\mathbb{D}$ centred at the origin. A {\it closed virtual twin diagram} of degree $n$ is an oriented virtual doodle diagram $K$ on the plane satisfying the following:
\begin{enumerate}
\item $K$ is contained in $\mathbb{R}^2 \setminus \mathbb{D}^{\circ}$.
\item If $\pi : \mathbb{R}^2 \setminus \mathbb{D}^{\circ} \to \mathbb{S}^1$ is the radial projection and $k : \sqcup~ \mathbb{S}^1 \to \mathbb{R}^2 \setminus \mathbb{D}^{\circ}$ the underlying immersion of $K$, then $$\pi \circ k: \sqcup~ \mathbb{S}^1 \to \mathbb{S}^1$$ is an $n$-fold covering, where $\mathbb{S}^1$ is the boundary of $\mathbb{D}$ and we assume it to be oriented counterclockwise.
\item The map $\pi$ restricted to $V(K)$, the set of all crossings of $K$,  is injective.
\item The orientation of $K$ is compatible with a fixed orientation of $\mathbb{S}^1$.
\end{enumerate}

Consider a point $p \in \mathbb{S}^1$ such that $\pi^{-1}(p) \cap V(K) = \phi $. Then cutting along the ray emanating from the origin and passing through $p$ gives a virtual twin diagram on $n$ strands. The \textit{closure} of a virtual twin diagram on the plane is defined to be the doodle obtained from the diagram by joining the end points with non-intersecting curves as shown in Figure \ref{Closure}. We note that there are many ways of taking closure of a virtual twin diagram.
\begin{figure}[hbtp]
\centering
\includegraphics[scale=0.6]{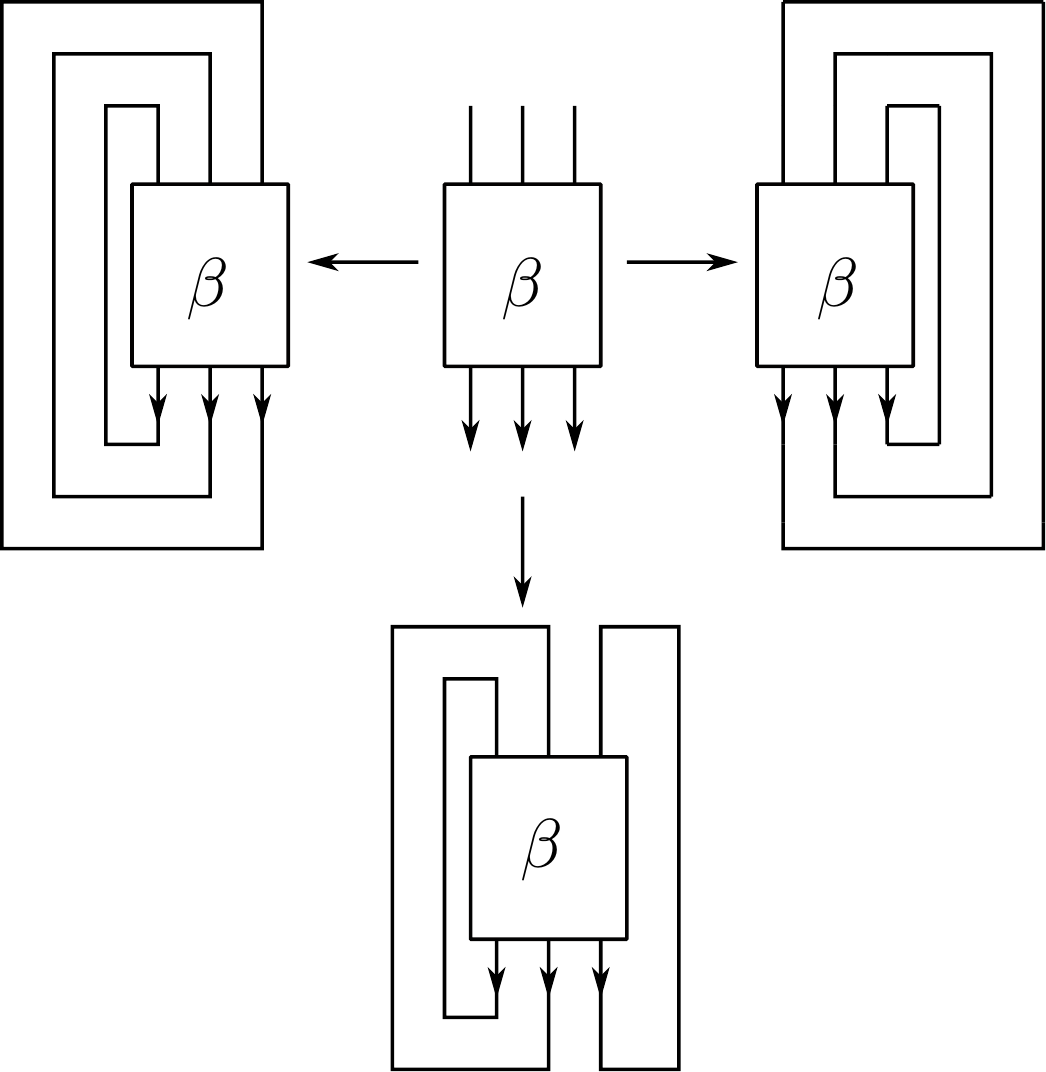}
\caption{Different closures of a virtual twin diagram}
\label{Closure}
\end{figure}

We observe that in the case of classical twins, due to forbidden move $s_is_{i+1}s_i \neq s_{i+1}s_1s_{i+1}$, taking closure of a twin diagram on a plane is not well-defined. The following result shows that the operation of taking closure on a plane in virtual setting is well-defined.

\begin{lemma}
Any two closures of a virtual twin diagram on the plane gives equivalent virtual doodle diagrams on the plane.
\end{lemma}

\begin{proof}
Let $\beta$ be a virtual twin diagram and $K$ and $K'$  two different closures of $\beta$. Then $K$ and $K'$ are a finite sequence of Kauffman's detour move depicted in Figure \ref{DetourMove}. By Lemma \ref{MainLemma}, $K$ and $K'$ are equivalent virtual doodle diagrams on the plane. 
\end{proof}

We now prove  Alexander Theorem for virtual doodles.

\begin{theorem}\label{Alexender theorem}
Every oriented virtual doodle on the plane is equivalent to closure of a virtual twin diagram.
\end{theorem}

\begin{proof}
Let $K$ be a virtual doodle diagram with $n$ real crossings. The idea is to construct a closed virtual twin diagram with the same Gauss data as that of $K$. The proof then follows from Lemma \ref{MainLemma}. We label each real crossing of $K$ as in Figure \ref{LabellingsOnRealCrossing}. Next, we consider $\mathbb{R}^2 \setminus \mathbb{D}^{\circ}$ and orient the boundary $\mathbb{S}^1$ of $\mathbb{D}$, say, counter clockwise. Considering the real crossings of $K$ with the information assigned as in Figure \ref{LabellingsOnRealCrossing}, we  place them in $\mathbb{R}^2 \setminus \mathbb{D}$ such that $\pi(c_i) \cap \pi(c_j) = \phi$ for all $i \neq j$ and the orientation is compatible with the orientation of $\mathbb{S}^1$.  Next, we join these crossings in $\mathbb{R}^2\setminus \mathbb{D}$ according to the Gauss data such that each intersection of arcs is marked as a virtual crossing and the orientation of arcs/loops are compatible with the orientation of $\mathbb{S}^1$, as illustrated in Figure \ref{AlexanderTheoremIllustration}. In other words, for each $(a,b) \in X(K)$ the orientation of the arc joining $a$ to $b$ should be compatible with the orientation of $\mathbb{S}^1$, that is, there is a possibility that we will have to wind the arc around $\mathbb{S}^1$ to join $a$ and $b$. Also, whenever it intersects with some other arc, then the intersection point should be marked as a virtual crossing. Note that this process is well defined upto detour moves shown in Figure \ref{DetourMove}, and virtual doodle so obtained is a closed virtual twin diagram which has the same Gauss data as that of $K$. Finally, cutting along $\pi^{-1} (p)$ for a point $ p \in \mathbb{S}^1$ such that $\pi^{-1} (p)$ does not pass through any crossing gives the desired virtual twin diagram whose closure is $K$. 
\end{proof}

\begin{figure}[hbtp]
\centering
\includegraphics[scale=.08]{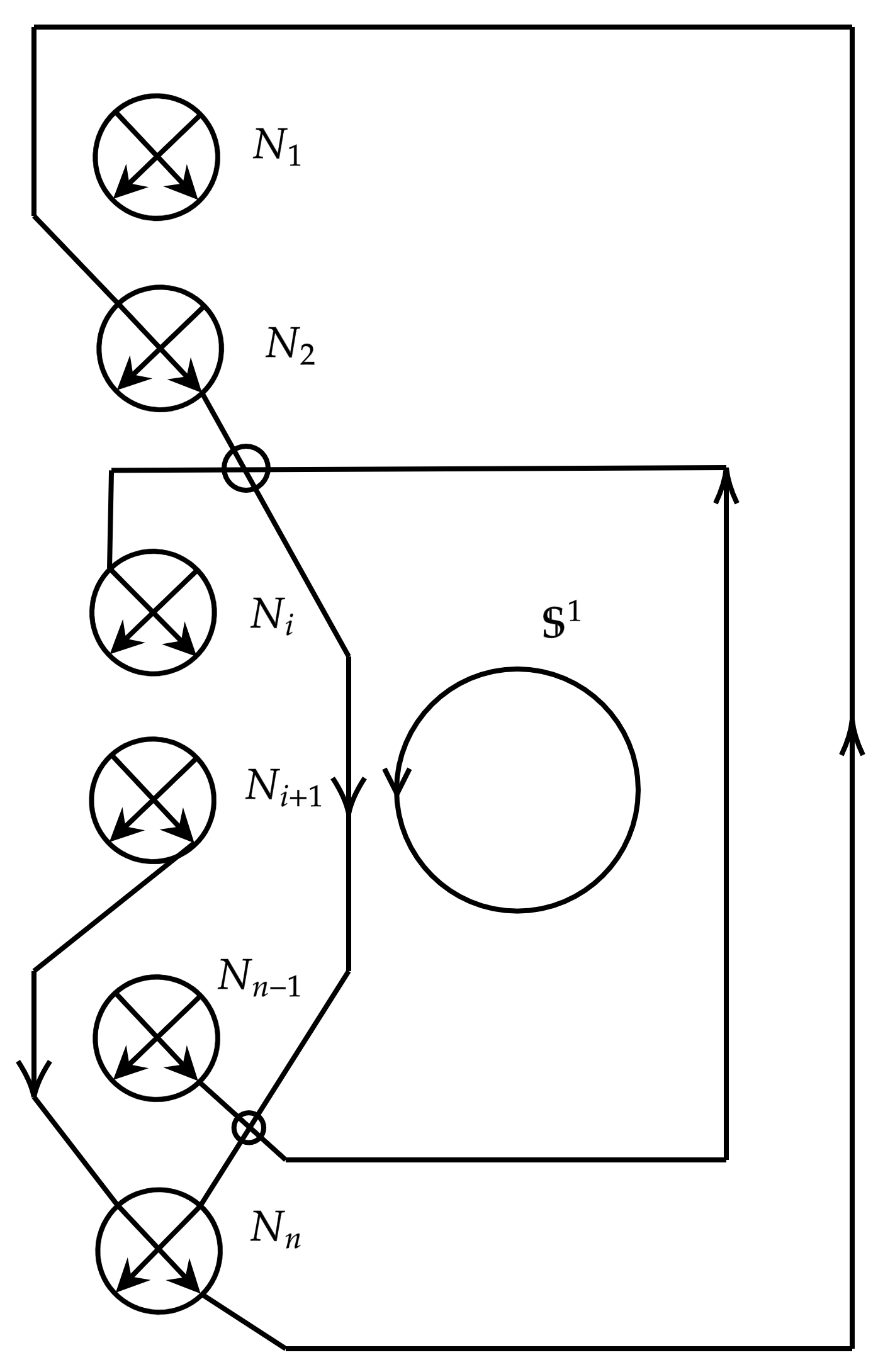}
\caption{}
\label{AlexanderTheoremIllustration}
\end{figure}

Following \cite{Kamada}, for convenience in writing, we refer the process of construction of a virtual twin in Theorem \ref{Alexender theorem} as the {\it braiding process} which is illustrated for virtual Kishino doodle in Figure \ref{AlexanderTheoremExampleKishino}.

\begin{figure}[hbtp]
\centering
\includegraphics[scale=.1]{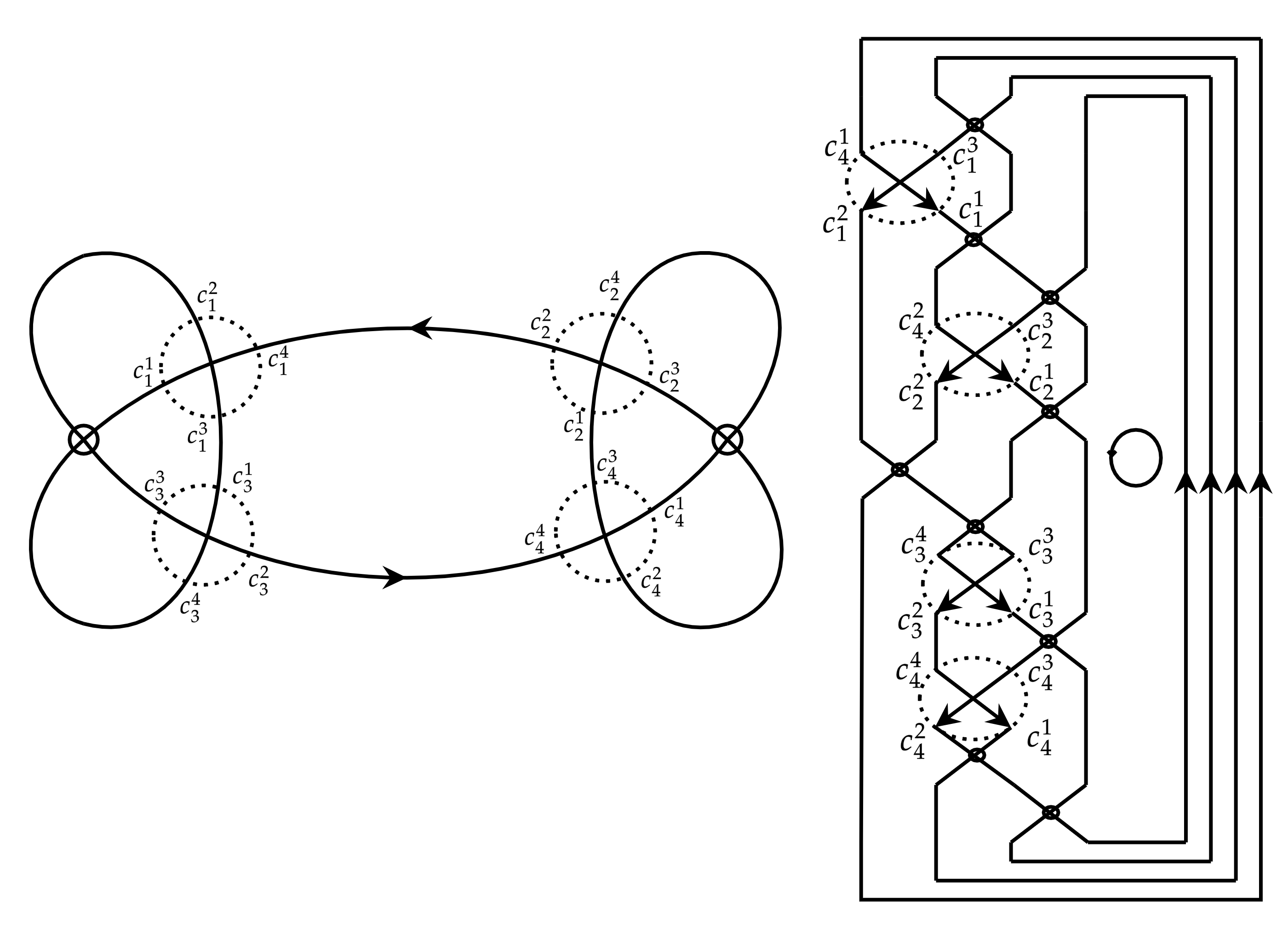}
\caption{Application of braiding process on virtual Kishino doodle}
\label{AlexanderTheoremExampleKishino}
\end{figure}

\section{Markov Theorem for virtual doodles}\label{Markov-theorem}
For $\beta \in VT_n$, let $m \otimes \beta  \in  VT_{n+m}$ denote the virtual twin obtained by putting trivial $m$ strands on the left of $\beta$. For $n \geq 2$ and virtual twins $\alpha, \beta, \beta_1, \beta_2 \in VT_n$, consider the following moves as illustrated in Figures \ref{StabilisationMoves} and \ref{ExchangeMoves}:
\begin{enumerate}
\item[$(M0)$] Defining relations \ref{virtual-twin-relations} in $VT_n$ (cf. Figure \ref{ReidemeisterMoves}).
\item[$(M1)$] Conjugation: $\alpha^{-1} \beta \alpha \sim \beta$.
\item[$(M2)$] Right stabilization of real or virtual type: $\beta s_n \sim \beta$ or $\beta \rho_n \sim \beta$.
\item[$(M3)$] Left stabilization of real type: $(1 \otimes \beta) s_1 \sim \beta$.
\item[$(M4)$] Right  exchange: $\beta_1 s_n \beta_2 s_n \sim \beta_1 \rho_n \beta_2 \rho_n$.
\item[$(M5)$] Left  exchange: $s_1 (1 \otimes \beta_1) s_1 (1 \otimes\beta_2) \sim \rho_1 (1 \otimes \beta_1) \rho_1 (1 \otimes \beta_2)$.
\end{enumerate}

We observe that the left stabilization of virtual type  $(1 \otimes \beta) \rho_1 \sim \beta$ is a consequence of the other moves as shown in Figure \ref{LeftStabilizationAsConsequence}. Note that the moves $M0-M5$ can be defined for closed virtual twin diagrams in a similar manner.

\begin{figure}[hbtp]
\centering
\includegraphics[scale=.25]{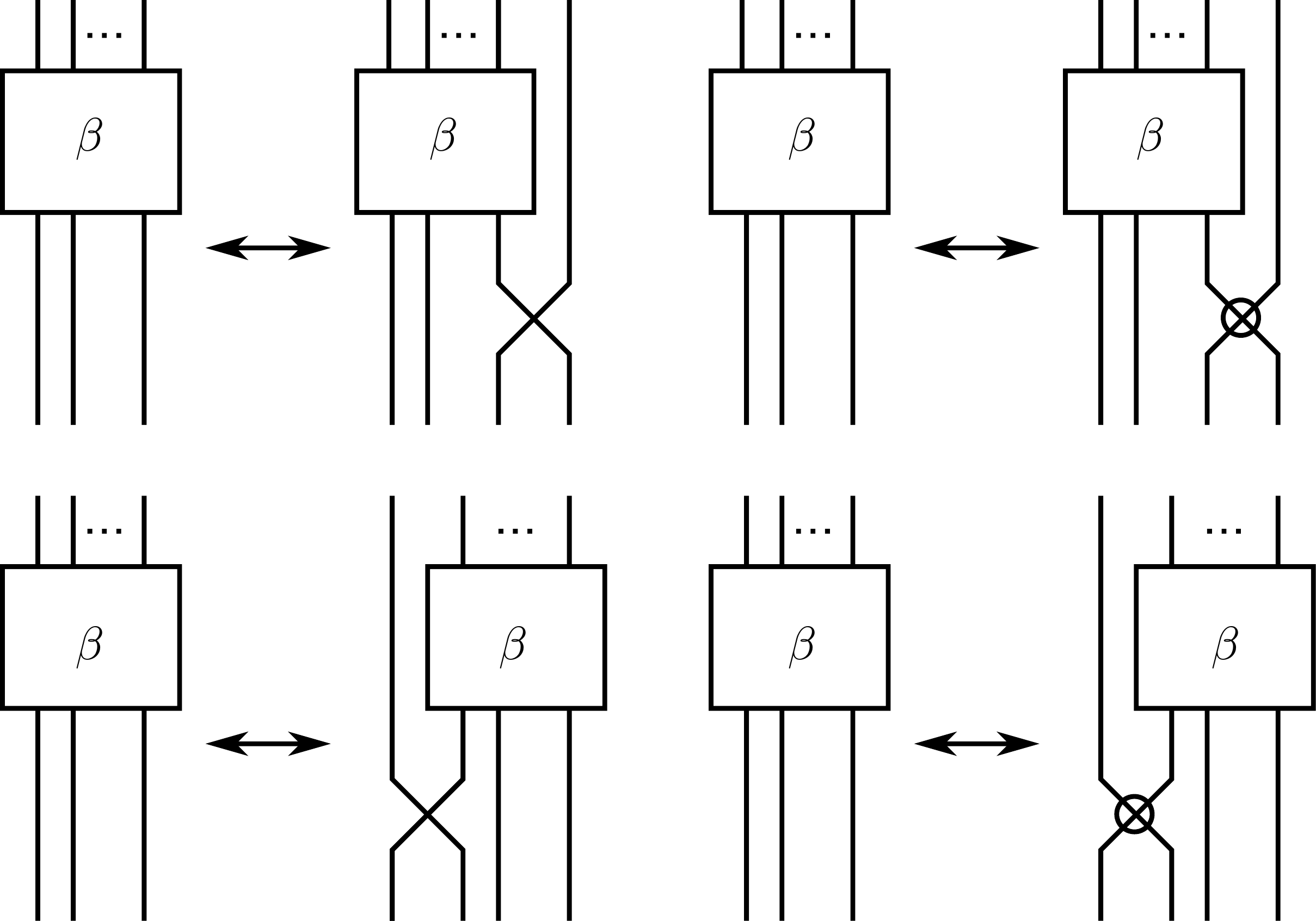}
\caption{Left and right stabilisation of real and virtual type}
\label{StabilisationMoves}
\end{figure}
\begin{figure}[hbtp]
\centering
\includegraphics[scale=0.25]{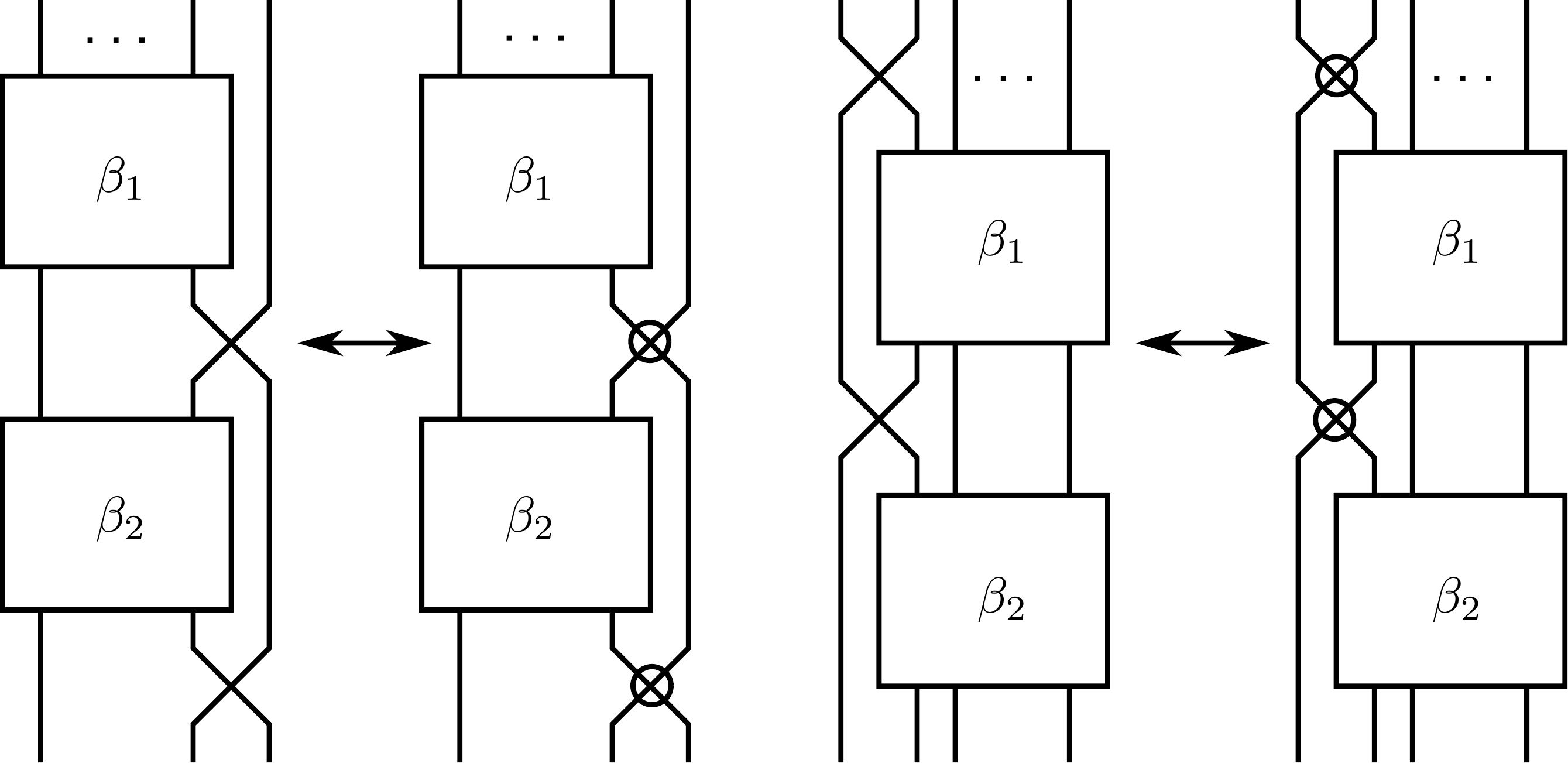}
\caption{Left and right exchange}
\label{ExchangeMoves}
\end{figure}

\begin{figure}[hbtp]
\centering
\includegraphics[scale=0.4]{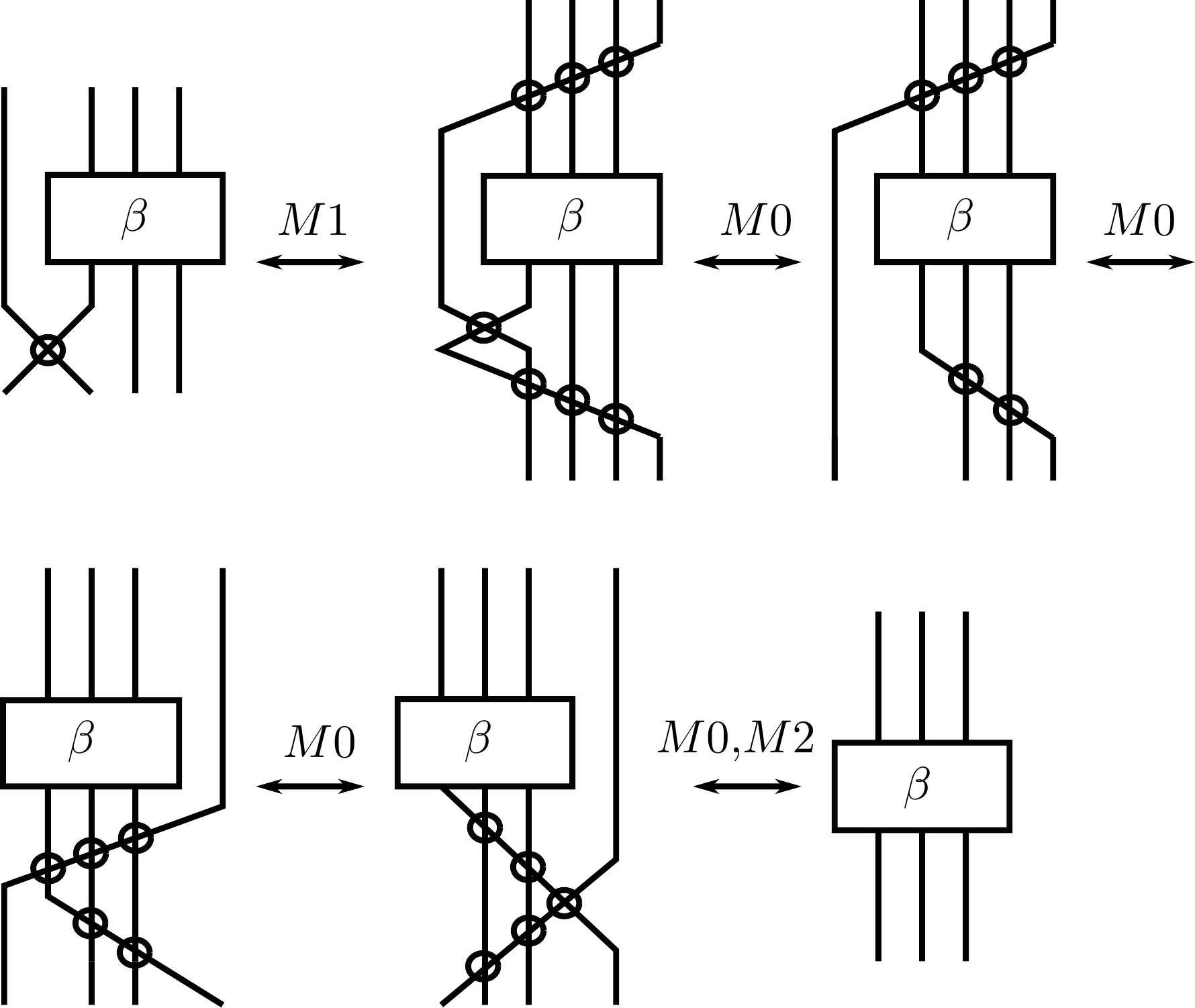}
\caption{Left stabilization of virtual type as a consequence of $M0-M5$}
\label{LeftStabilizationAsConsequence}
\end{figure}

\begin{lemma}\label{Lemma1}
Let $ n \geq 2$ and $1 \leq i \leq n$. Under the assumption of moves $M0-M5$, the following hold:
\begin{enumerate}
\item $\beta s_n s_{n-1} \dots s_{i+1}s_is_{i+1} \dots s_{n-1}s_{n} \sim \beta$, where $\beta \in VT_n$.
\item $s_n s_{n-1} \dots s_{i+1}s_i \beta_1s_i s_{i+1} \dots s_n \beta_2 \sim   \rho_n \rho_{n-1} \dots \rho_{i+1} \rho_i \beta_1 \rho_i \rho_{i+1} \dots \rho_n \beta_2$, where $\beta_1 \in VT_{i}$ and $\beta_2 \in VT_n$.
\item $\tau_n \tau_{n-1} \dots \tau_{i+1} \tau_i \beta_1 \tau_i \tau_{i+1} \dots \tau_{n-1} \tau_n \beta_2 \sim   \rho_n \rho_{n-1} \dots \rho_{i+1} \rho_i \beta_1 \rho_i \rho_{i+1} \dots \rho_n \beta_2$, where $\beta_1 \in VT_{i}$, $\beta_2 \in VT_n$ and  $\tau_j=s_j$ or $\rho_j$ for each $j$.
\item  $\beta \tau_n \tau_{n-1} \dots \tau_{i} \tau_{i-1} \tau_{i} \dots \tau_{n-1} \tau_{n} \sim \beta$, where $\beta \in VT_n$ and $\tau_j=s_j$ or $\rho_j$ for each $j$.
\end{enumerate} 
\end{lemma}

\begin{proof}
We begin by observing that the case $i=n$ holds due to move $M2$. Also, for $i=n-1$, we have 
\begin{eqnarray*}
\beta \underline{s_n} s_{n-1} \underline{s_n} &\stackrel{M4}{\sim}& \beta \underline{\rho_n s_{n-1} \rho_n}\\
&\stackrel{M0}{\sim}& \beta \rho_{n-1} s_{n} \underline{\rho_{n-1}}\\
&\stackrel{M1}{\sim}&  \rho_{n-1} \beta \rho_{n-1} \underline{s_{n}}\\
&\stackrel{M2}{\sim}&  \underline{\rho_{n-1}} \beta \underline{\rho_{n-1}}\\
&\stackrel{M1}{\sim}&  \beta.
\end{eqnarray*}

Let us suppose that 
\begin{equation}\label{eqn1}
\beta s_n s_{n-1} \dots s_{i+2}s_{i+1}s_{i+2} \dots s_{n-1}s_{n} \sim \beta
\end{equation}
for $1 \leq i \leq n-2$ and for any $\beta \in VT_n$.
Then, we have 
\begin{eqnarray*}
& & \beta \underline{s_n} s_{n-1} \dots s_{i+1}s_is_{i+1} \dots s_{n-1}\underline{s_{n}}\\
 &\stackrel{M4}{\sim}& \beta \rho_n s_{n-1} \underline{s_{n-2} \dots s_{i+1}s_is_{i+1} \dots s_{n-2}} s_{n-1}\rho_{n}\\
 &\stackrel{M0}{\sim}& \beta \underline{\rho_n s_{n-1} \rho_n} \dots s_{i+1}s_is_{i+1} \dots \underline{\rho_n s_{n-1}\rho_{n}}\\
 &\stackrel{M0}{\sim}&\beta \rho_{n-1} s_{n} \rho_{n-1}\underline{s_{n-2} \dots s_{i+1}s_is_{i+1} \dots s_{n-2}}\rho_{n-1} s_{n}\rho_{n-1}\\
 &\stackrel{M0}{\sim}&\beta \rho_{n-1} s_{n} \underline{\rho_{n-1} s_{n-2} \rho_{n-1}} \dots s_{i+1}s_i s_{i+1} \dots \underline{\rho_{n-1} s_{n-2}\rho_{n-1}} s_{n} \rho_{n-1}\\
 &\stackrel{M0}{\sim}& \beta \rho_{n-1} \underline{s_{n} \rho_{n-2}} s_{n-1} \rho_{n-2} \dots s_{i+1}s_i s_{i+1} \dots \rho_{n-2} s_{n-1}\underline{\rho_{n-2} s_{n}} \rho_{n-1}\\
 &\stackrel{M0}{\sim}& \beta \rho_{n-1}  \rho_{n-2} s_{n} s_{n-1} \rho_{n-2} \dots s_{i+1}s_i s_{i+1} \dots \rho_{n-2} s_{n-1} s_{n} \rho_{n-2} \rho_{n-1}.
\end{eqnarray*}

Repeating the above steps give
\begin{eqnarray*}
& & \beta s_n s_{n-1} \dots s_{i+1}s_is_{i+1} \dots s_{n-1}s_{n}\\
&\sim& \beta \rho_{n-1}  \rho_{n-2} \dots \rho_{i+1} s_{n} s_{n-1} \dots s_{i+2} \underline{\rho_{i+1} s_i \rho_{i+1}} s_{i+2} \dots s_{n-1} s_{n} \rho_{i+1} \dots\rho_{n-2} \rho_{n-1}\\
&\stackrel{M0}{\sim}& \beta \rho_{n-1}  \rho_{n-2} \dots \rho_{i+1} s_{n} s_{n-1} \dots s_{i+2}\underline{\rho_{i}} s_{i+1} \underline{\rho_{i}} s_{i+2} \dots s_{n-1} s_{n} \rho_{i+1} \dots\rho_{n-2} \rho_{n-1}\\
&\stackrel{M0}{\sim}& \beta \rho_{n-1} \rho_{n-2} \dots \rho_i s_n s_{n-1} \dots s_{i+2} s_{i+1} s_{i+2} \dots s_{n-1}s_n \underline{\rho_i \dots \rho_{n-2} \rho_{n-1}}\\
&\stackrel{M1}{\sim}& \rho_i \dots\rho_{n-2} \rho_{n-1}\beta \rho_{n-1} \rho_{n-2} \dots \rho_i \underline{s_n s_{n-1} \dots s_{i+2} s_{i+1} s_{i+2} \dots s_{n-1} s_n}.
\end{eqnarray*}

Since $\rho_i \dots\rho_{n-2} \rho_{n-1}\beta \rho_{n-1} \rho_{n-2} \dots \rho_i \in VT_n$, by (\ref{eqn1}) and move $M1$, we get
$$ \beta s_n s_{n-1} \dots s_{i+1}s_is_{i+1} \dots s_{n-1}s_{n} \stackrel{(6.0.1)}{\sim} \underline{\rho_i \dots\rho_{n-2} \rho_{n-1}}\beta \underline{\rho_{n-1} \rho_{n-2} \dots \rho_i} \stackrel{M1}{\sim} \beta.$$
This proves assertion (1).

For assertion (2), note that the case $i=n$ follows from  moves $M1$ and $M4$. Let us suppose that for any $\beta_1 \in VT_{i+1}$ and $\beta_2 \in VT_n$, we have
\begin{equation}\label{induction2}
s_n s_{n-1} \dots s_{i+2}s_{i+1} \beta_1s_{i+1} s_{i+2} \dots s_n \beta_2 \sim   \rho_n \rho_{n-1} \dots \rho_{i+2} \rho_{i+1} \beta_1 \rho_{i+1} \rho_{i+2} \dots \rho_n \beta_2.
\end{equation}
We claim that 
$$s_n s_{n-1} \dots s_{i+1}s_i \beta_1s_i s_{i+1} \dots s_{n-1}s_n \beta_2 \sim   \rho_n \rho_{n-1} \dots \rho_{i+1} \rho_i \beta_1 \rho_i \rho_{i+1} \dots \rho_{n-1} \rho_n \beta_2$$ for  $\beta_1 \in VT_{i}$ and $\beta_2 \in VT_n$.
For $1 \leq i \leq n-1$, we have
\begin{eqnarray*}
&&s_n s_{n-1} \dots s_{i+1}s_i \beta_1s_i s_{i+1} \dots s_{n-1}s_n \underline{\beta_2}\\
 &\stackrel{M1}{\sim} & \beta_2 \underline{s_n} s_{n-1} \dots s_{i+1}s_i \beta_1s_i s_{i+1} \dots s_{n-1}\underline{s_n}\\
 &\stackrel{M4}{\sim} & \underline{\beta_2} \rho_n s_{n-1} \dots s_{i+1}s_i \beta_1s_i s_{i+1} \dots s_{n-1} \rho_n\\
 &\stackrel{M1}{\sim} &\rho_n s_{n-1}\underline{s_{n-2} \dots s_{i+1} s_i \beta_1 s_i s_{i+1} \dots s_{n-2}}s_{n-1} \rho_n \beta_2\\
 &\stackrel{M0}{\sim} & \underline{\rho_n s_{n-1}\rho_n} \dots s_{i+1} s_i \beta_1 s_i s_{i+1} \dots \underline{\rho_n s_{n-1} \rho_n} \beta_2\\
 &\stackrel{M0}{\sim} & \rho_{n-1} s_{n}\rho_{n-1} \dots s_{i+1} s_i \beta_1 s_i s_{i+1} \dots \rho_{n-1}s_{n} \rho_{n-1} \beta_2\\
 &\stackrel{M0}{\sim} & \rho_{n-1} s_{n}\underline{\rho_{n-1}s_{n-2}\rho_{n-1}} \dots s_i \beta_1 s_i \dots \underline{\rho_{n-1}s_{n-2}\rho_{n-1}}s_{n} \rho_{n-1} \beta_2\\
 &\stackrel{M0}{\sim} & \rho_{n-1} \underline{s_{n}\rho_{n-2}}s_{n-1}\rho_{n-2} \dots s_i \beta_1 s_i \dots \rho_{n-2}s_{n-1}\underline{\rho_{n-2}s_{n}} \rho_{n-1} \beta_2\\
 &\stackrel{M0}{\sim} & \rho_{n-1} \rho_{n-2}s_{n}s_{n-1}\rho_{n-2} \dots s_i \beta_1 s_i \dots \rho_{n-2}s_{n-1}s_{n}\rho_{n-2} \rho_{n-1} \beta_2.
 \end{eqnarray*}

Repeating the preceding process yields
\begin{eqnarray*}
& & s_n s_{n-1} \dots s_{i+1}s_i \beta_1s_i s_{i+1} \dots s_{n-1}s_n \beta_2\\
&\sim &  \rho_{n-1}  \rho_{n-2} \dots \rho_{i} \underline{s_{n} s_{n-1} \dots s_{i+1} \rho_{i} \beta_1 \rho_i s_{i+1} \dots s_{n-1} s_{n}} \rho_{i} \dots\rho_{n-2} \rho_{n-1}\beta_2.
 \end{eqnarray*}

Notice that $\rho_i\beta_1\rho_i \in VT_{i+1}$ and $\rho_i \dots \rho_{n-2}\rho_{n-1} \beta_2 \rho_{n-1}\rho_{n-2} \dots \rho_i \in VT_n$. By (\ref{induction2}) and $M1$, we get 
\begin{eqnarray*}
&& s_n s_{n-1} \dots s_{i+1}s_i \beta_1s_i s_{i+1} \dots s_{n-1}s_n \beta_2\\
&\sim&  \underline{\rho_{n-1}  \rho_{n-2} \dots \rho_{i}}s_{n} s_{n-1} \dots s_{i+1} \rho_{i} \beta_1 \rho_i s_{i+1} \dots s_{n-1} s_{n} \rho_{i} \dots\rho_{n-2} \rho_{n-1}\beta_2\\
&\stackrel{M1}{\sim} & (s_{n} s_{n-1} \dots s_{i+1})( \rho_{i} \beta_1 \rho_i )(s_{i+1} \dots s_{n-1} s_{n})(\rho_{i} \dots\rho_{n-2} \rho_{n-1}\beta_2 \rho_{n-1}  \rho_{n-2} \dots \rho_{i})\\
&\stackrel{(6.0.2)}{\sim} & (\rho_{n} \rho_{n-1} \dots \rho_{i+1})( \rho_{i} \beta_1 \rho_i )(\rho_{i+1} \dots \rho_{n-1} \rho_{n})(\rho_{i} \dots\rho_{n-2} \rho_{n-1}\beta_2 \underline{\rho_{n-1}  \rho_{n-2} \dots \rho_{i}})\\
&\stackrel{M1}{\sim} &  \rho_{n-1}  \rho_{n-2} \dots \underline{\rho_{i}} \rho_{n} \rho_{n-1} \dots \rho_{i+1} \rho_{i} \beta_1 \rho_i \rho_{i+1} \dots \rho_{n-1} \rho_{n} \underline{\rho_{i}} \dots\rho_{n-2} \rho_{n-1}\beta_2\\
&\stackrel{M0}{\sim} &  \rho_{n-1}\rho_{n-2}  \dots\rho_{i+1} \rho_{n} \rho_{n-1} \dots \underline{\rho_{i}\rho_{i+1} \rho_{i}} \beta_1 \underline{\rho_i \rho_{i+1} \rho_{i}} \dots \rho_{n-1} \rho_{n} \rho_{i+1} \dots  \rho_{n-2}\rho_{n-1}\beta_2\\
&\stackrel{M0}{\sim} &  \rho_{n-1} \rho_{n-2} \dots\rho_{i+1} \rho_{n} \rho_{n-1} \dots \rho_{i+1}\rho_{i} \underline{\rho_{i+1} \beta_1 \rho_{i+1}} \rho_{i} \rho_{i+1} \dots \rho_{n-1} \rho_{n} \rho_{i+1} \dots   \rho_{n-1}\beta_2\\
&\sim& \rho_{n-1}  \dots\underline{\rho_{i+1}} \rho_{n} \rho_{n-1} \dots \rho_{i+1}\rho_{i} \beta_1 \rho_{i} \rho_{i+1} \dots \rho_{n-1} \rho_{n} \underline{\rho_{i+1}}\dots\rho_{n-1}\beta_2,\\
&&\textrm{($\rho_{i+1}$'s gets canceled as $\beta_1 \in VT_i$)}.
\end{eqnarray*}

Repeating the above steps finally gives
$$s_n s_{n-1} \dots s_{i+1}s_i \beta_1s_i s_{i+1} \dots s_{n-1}s_n \beta_2 \sim   \rho_n \rho_{n-1} \dots \rho_{i+1} \rho_i \beta_1 \rho_i \rho_{i+1} \dots \rho_{n-1} \rho_n \beta_2,$$ which proves assertion (2).

Repeatedly applying (2) on the expression $\tau_n \tau_{n-1} \dots \tau_{i+1} \tau_i \beta_1 \tau_i \tau_{i+1} \dots \tau_{n-1} \tau_n \beta_2$ yields assertion (3). For example,
\begin{eqnarray*}
&& \underline{s_n} \rho_{n-1} s_{n-2} \rho_{n-3} \beta_1 \rho_{n-3} s_{n-2} \rho_{n-1} \underline{s_n} \beta_2\\
 &\sim& \underline{\rho_n \rho_{n-1}} s_{n-2} \rho_{n-3} \beta_1 \rho_{n-3} s_{n-2} \underline{\rho_{n-1} \rho_n} \beta_2\\
 &\sim& \underline{s_n s_{n-1} s_{n-2}} \rho_{n-3} \beta_1 \rho_{n-3} \underline{s_{n-2} s_{n-1} s_n} \beta_2\\
 &\sim& \rho_n \rho_{n-1} \rho_{n-2} \rho_{n-3} \beta_1 \rho_{n-3} \rho_{n-2} \rho_{n-1} \rho_n \beta_2.
\end{eqnarray*}

For assertion (4), if we put $\beta_1=\tau_{i-1}$ and $\beta_2=\beta$ in assertion (3), then we get 
\begin{eqnarray*}
&& \underline{\tau_n \tau_{n-1} \dots \tau_{i+1} \tau_i \tau_{i-1} \tau_i \tau_{i+1} \dots \tau_{n-1} \tau_n} \beta\\
&\stackrel{M1}{\sim} & \beta \underline{\tau_n \tau_{n-1} \dots \tau_{i+1} \tau_i \tau_{i-1} \tau_i \tau_{i+1} \dots \tau_{n-1} \tau_n}\\ 
& \sim &  \beta \rho_n \rho_{n-1} \dots \rho_{i+1} \rho_i \tau_{i-1} \rho_i \rho_{i+1} \dots \rho_{n-1} \rho_n ~~\text{\big(by taking $\beta_1 = \tau_{i-1}$ and $\beta_2 = \beta$ in (3)\big)}.\\
\end{eqnarray*}

If $\tau=\rho$, then 
\begin{eqnarray*}
&&\beta \rho_n \rho_{n-1} \dots \rho_{i} \rho_{i-1} \rho_{i} \dots \rho_{n-1} \rho_{n}\\
&\sim &  \beta \underline{\rho_n \rho_{n-1} \dots \rho_{i-1}\rho_{i} \rho_{i+2} \rho_{i+1} \rho_{i+2} \rho_{i}\rho_{i-1} \dots \rho_{n-1} \rho_{n}}~~\textrm{(by repeated application of $M0$)}\\
&\sim & \beta \rho_{i-1} \rho_{i} \dots \rho_{n-1} \rho_{n} \underline{\rho_{n-1} \dots \rho_{i} \rho_{i-1}}~~\textrm{(by repeated application of the preceding step)}\\ 
&\stackrel{M1}{\sim} & \rho_{n-1} \dots \rho_{i} \rho_{i-1} \beta \rho_{i-1} \rho_{i} \dots \rho_{n-1} \underline{\rho_{n}}\\ 
&\stackrel{M2}{\sim} & \underline{\rho_{n-1} \dots \rho_{i} \rho_{i-1}} \beta \underline{\rho_{i-1} \rho_{i} \dots \rho_{n-1}}\\ 
&\stackrel{M1}{\sim} & \beta.
\end{eqnarray*}

Finally if  $\tau=s$, then we get
\begin{eqnarray*}
&& \beta \underline{\rho_n \rho_{n-1} \dots \rho_{i} s_{i-1} \rho_{i} \dots \rho_{n-1} \rho_{n}}\\
&\stackrel{(2)}{\sim} & \beta \underline{s_n s_{n-1} \dots s_{i} s_{i-1} s_{i} \dots s_{n-1} s_{n}}\\
&\stackrel{(1)}{\sim} & \beta,
\end{eqnarray*}
which completes the proof.
\end{proof}
 
Recall that for $\beta \in VT_n$, $m \otimes \beta  \in  VT_{n+m}$ denotes the virtual twin obtained by putting trivial $m$ strands on the left of $\beta$.
\begin{lemma}\label{Lemma2}
Let $ n \geq 2$ and $1 \leq i \leq n$. Under the assumption of moves $M0-M5$, the following hold:
\begin{enumerate}
\item $(1 \otimes \beta) s_1 s_{2} \dots s_{i-1}s_is_{i-1} \dots s_{2}s_{1} \sim \beta$, where $\beta \in VT_n$.
\item $s_1 s_{2} \dots s_{i-1}s_i (i \otimes \beta_1)s_i s_{i-1} \dots s_{2}s_1 (1 \otimes \beta_2) \sim \\ \rho_1 \rho_{2} \dots \rho_{i-1}\rho_i (i \otimes \beta_1)\rho_i \rho_{i-1} \dots \rho_{2}\rho_1 (1 \otimes \beta_2)$, where $\beta_1 \in VT_{n+1-i}$ and $\beta_2 \in VT_n$.
\item $\tau_1 \tau_{2} \dots \tau_{i-1} \tau_i  (i \otimes \beta_1) \tau_i \tau_{i-1} \dots \tau_{2} \tau_1   (1 \otimes \beta_2) \sim \\ \rho_1 \rho_{2} \dots \rho_{i-1} \rho_i  (i \otimes \beta_1) \rho_i \rho_{i-1} \dots \rho_{2} \rho_1   (1 \otimes \beta_2)$, where $\beta_1 \in VT_{n+1-i}$, $\beta_2 \in VT_n$ and $\tau_j=s_j$ or $\rho_j$ for each $j$.
\item  $ (1 \otimes \beta) \tau_1 \tau_{2} \dots \tau_{i-1} \tau_{i} \tau_{i-1} \dots \tau_{2} \tau_{1} \sim \beta$, where $\beta \in VT_n$ and $\tau_j=s_j$ or $\rho_j$ for each $j$.
\end{enumerate}
\end{lemma}

\begin{proof}
The proof is similar to that of Lemma \ref{Lemma1}.
\end{proof}

Recall that for a virtual doodle diagram $K$ on the plane, $W(K)$ denotes the closure of the complement of union of closed disk neighbourhoods of real crossings of $K$. The proofs of the following two lemmas are similar to \cite[Lemma 5 and Lemma 6]{Kamada}. We give proofs in our setting for the sake of completeness.

\begin{lemma}\label{Lemma3}
Let $K$ and $K'$ be two closed virtual twin diagrams such that $K'$ is obtained from $K$ by replacing $K \cap W(K)$ by $K' \cap W(K')$. Then $K$ and $K'$ are related by a finite sequence of $M0$ and $M2$ moves.
\end{lemma}

\begin{proof}
We use notation from sections \ref{virtual-doodle-diagram} and \ref{Alexander-theorem-section}. Let $\pi$ be the radial projection. Let $N_1, N_2, \dots, N_n$ be  closed $2$-disks enclosing real crossings of $K$ and hence of $K'$ such that $\pi(N_i) \cap \pi(N_j)=\phi$ for all $i \neq j$, that is, real crossings lie at separate levels. Let $a_1, a_2, \dots, a_s$ be arcs/loops in $K \cap W(K)$ and $a_1', a_2', \dots, a_s'$ be the corresponding arcs/loops in $K' \cap W(K')$. Consider a  point $p \in \mathbb{S}^1$ such that $\pi^{-1}(p)$ does not intersect either of the crossing sets $V(K)$ and $V(K')$. If there exists some arc/loop $a_i$ and its corresponding arc/loop $a_i'$ such that $|a_i \cap \pi^{-1}(p)| \neq |a_i' \cap \pi^{-1}(p)|$, then we bring a segment of $a_i$ or $a_i'$ closer to the origin by repeated use of $\rho_i^2=1$ and some $M2$ moves of virtual type such that $|a_i \cap \pi^{-1}(p)|= |a_i' \cap \pi^{-1}(p)|$. Thus, we can assume that $|a_i \cap \pi^{-1}(p)| = |a_i' \cap \pi^{-1}(p)|$ for all $i$.
\par
 
 Let $k$ and $k'$ be the underlying immersions $\sqcup~ \mathbb{S}^1 \to\mathbb{R}^2 \setminus \mathbb{D}^{\circ}$ of $K$ and $K'$, respectively, such that they are identical in preimage of each $N_i$. Let $I_1, I_2, \dots, I_s$ be intervals/circles in $\sqcup~ \mathbb{S}^1$ such that $k(I_i)=a_i$ and $k'(I_i)=a_i'$. We note that $\pi \circ k \restr{I_i}$ and $\pi \circ k' \restr{I_i}$ are orientation preserving immersions with $\pi \circ k \restr{\partial I_i}=\pi \circ k' \restr{\partial I_i}$. Since $|a_i \cap \pi^{-1}(p)|= |a_i' \cap \pi^{-1}(p)|$ for any $i$, there exists a homotopy $k^t_i : I_i \to \mathbb{R}^2 \setminus \mathbb{D}^{\circ}$ relative to boundary $\partial I_i$ such that $k^0_i= k \restr{I_i}$ and $k^1_i= k' \restr{I_i}$ and  $\pi \circ k^t_i$ is an orientation preserving immersion. If we take the homotopy generically with respect to $K \cap W(K)$, $K' \cap W(K')$ and the $2$-disks $N_j$, we see that $a_i'$ can be transformed to $a_i$ by a sequence of $VR_2$, $VR_3$ and $M$ moves in $\mathbb{R}^2 \setminus \mathbb{D}^{\circ}$. Consequently, $K$ and $K'$ are related by a finite sequence of $M0$ and $M2$ moves.
\end{proof}

\begin{lemma} \label{Lemma4}
Let $K$ and $K'$ be closed virtual twin diagrams having the same Gauss data. Then $K$ and $K'$ are related by a finite sequence of $M0$ and $M2$ moves.
\end{lemma}

\begin{proof}
Let $N_1, N_2, \dots, N_n$ be closed $2$-disks enclosing real crossings of $K$ and $N_1', N_2', \dots, N_n'$ be the corresponding closed $2$-disks enclosing real crossings of $K'$. We consider two cases depending on the position of $N_i$ and $N_j'$ with respect to the map $\pi$.

\noindent {{Case I.}} Suppose that $\pi(N_1), \pi(N_2), \dots, \pi(N_n)$ and $\pi(N_1'), \pi(N_2'), \dots, \pi(N_n')$ appear in the same cyclic order on boundary $\mathbb{S}^1$. Then we deform $K$  by isotopies of the plane such that $N_i=N_i'$ for all $i$ and diagrams of $K$ and $K'$ are identical in $N_i$ for all $i$. Thus, $K'$ can be obtained from $K$ by replacing $K \cap W(K)$ by $K' \cap W(K')$, and we are done by Lemma \ref{Lemma3}.
\

\noindent {{Case II.}} Suppose that $\pi(N_1), \pi(N_2), \dots, \pi(N_n)$ and $\pi(N_1'), \pi(N_2'), \dots, \pi(N_n')$ do not appear in the same cyclic order on $\mathbb{S}^1$. Without loss of generality, we may assume that the two sequences of sets appear in the same order except $\pi(N_1)$ and $\pi(N_2)$. Notice that the diagram $K$ looks as shown in the leftmost part in Figure \ref{LemmaDiagram}, where $\beta_1$ is a virtual twin diagram with no real crossing and $\beta_2$ a virtual twin diagram. As shown in Figure \ref{LemmaDiagram}, we can make $\pi(N_1), \pi(N_2), \dots, \pi(N_n)$ and $\pi(N_1'), \pi(N_2'), \dots, \pi(N_n')$ to appear in the same cyclic order on $\mathbb{S}^1$ using $M0$ and $M2$ moves. Thus, we get back to Case I and we are done.
\begin{figure}[hbtp]
\centering
\includegraphics[scale=0.3]{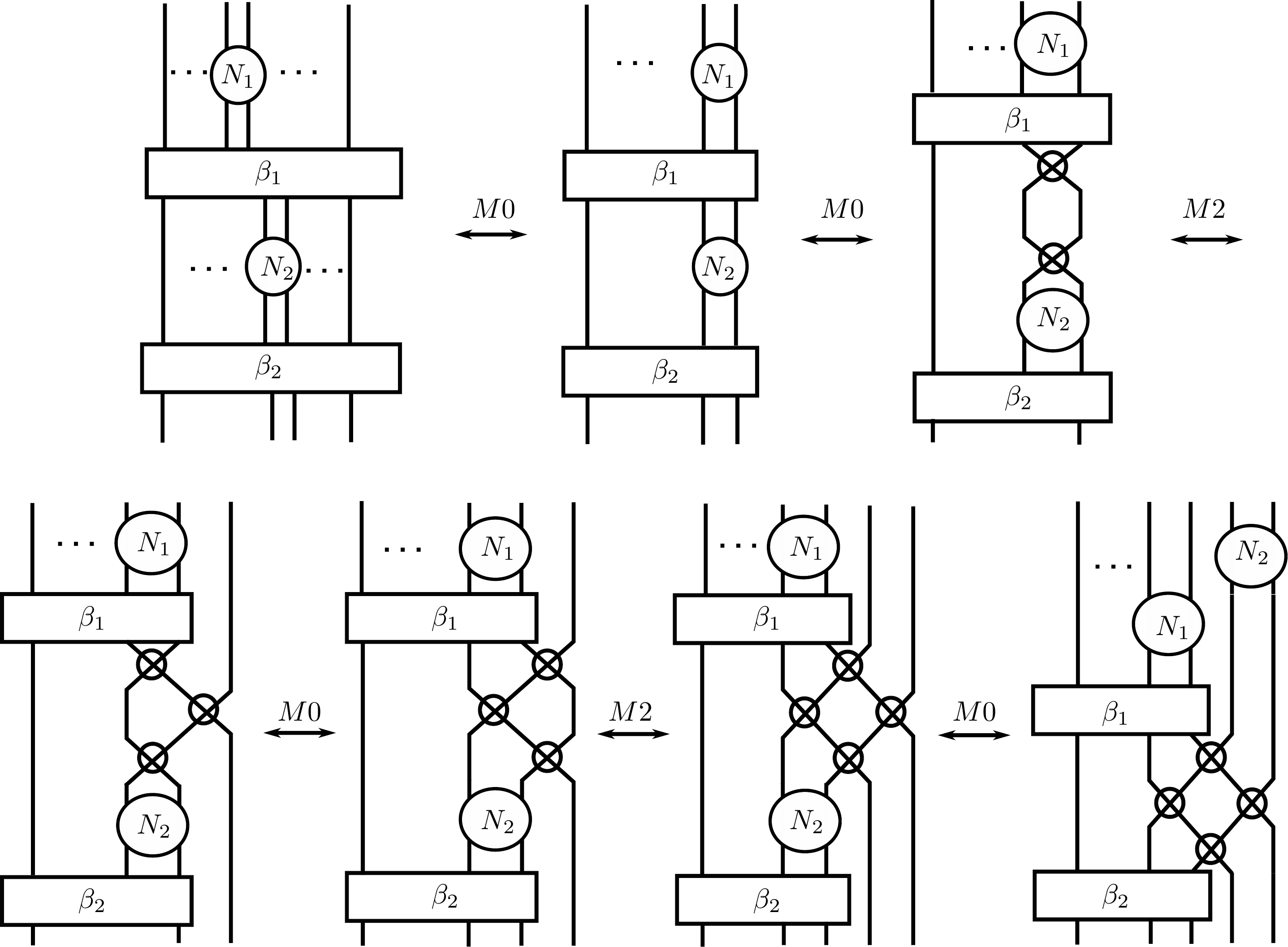}
\caption{}
\label{LemmaDiagram}
\end{figure}
\end{proof}

\begin{corollary}\label{Closed Twin and Doodle}
A closed virtual twin diagram for any oriented virtual doodle is uniquely determined upto $M0$ and $M2$ moves.
\end{corollary}

\begin{proof}
It follows from the fact that any two closed virtual twin diagrams for a virtual doodle have the same Gauss data (as in the proof of Theorem \ref{Alexender theorem}). The result then follows from Lemma \ref{Lemma4}.
\end{proof}

We now state and prove Markov Theorem for virtual doodles.

\begin{theorem}\label{Markov Theorem}
Two virtual twin diagrams on the plane  (possibly on different number of strands) have equivalent closures if and only if they are related by a finite sequence of moves $M0-M5$.
\end{theorem}

\begin{proof}
The proof of the converse implication is immediate. For the forward implication, let $K$ and $K'$ be two closed virtual twin diagrams which are equivalent as virtual doodles. That is, there is a finite sequence of virtual doodle diagrams, say, $K=K_0, K_1, \dots , K_{n}=K'$ such that $K_i$ is obtained from $K_{i-1}$ by one of the moves as shown in Figure \ref{OrientedReidemeisterMoves}. Note that the virtual doodle diagrams obtained in the intermediate steps may not be closed virtual twin diagrams. Let $\widetilde{K}_i$ be a closed virtual twin diagram for $K_i$ obtained by the braiding process as in the proof of Theorem \ref{Alexender theorem}. Without loss of generality,  we can assume that $\widetilde{K}_0=K_0$ and $\widetilde{K}_n=K_n$. By Corollary \ref{Closed Twin and Doodle}, we know that each $\widetilde{K}_i$ is uniquely determined up to $M0$ and $M2$ moves. Thus, it suffices to prove that $\widetilde{K}_{i-1}$ and $\widetilde{K}_{i}$ are related by $M0-M5$ moves. We proceed by considering each move in Figure \ref{OrientedReidemeisterMoves}.

\noindent {{Case I.}} Let $K_{i}$ be obtained from $K_{i-1}$ by applying any one of the $VR_1$, $VR_2$, $VR_3$ or $M$ moves. Then $K_{i}$ and $K_{i-1}$ have the same Gauss data, which means that $\widetilde{K}_{i}$ and $\widetilde{K}_{i-1}$ also have the same Gauss data. Then, by Lemma \ref{Lemma4}, $\widetilde{K}_{i-1}$ and $\widetilde{K}_{i}$ are related by $M0$ and $M2$  moves.

\noindent {{Case II.}} If $K_{i}$ is obtained from $K_{i-1}$ by an $R_2$ move, then $\widetilde{K}_{i-1}$ and $\widetilde{K}_{i}$ are related by a $M0$ move and we are done.
\medskip

For the remaining moves, let $\mathbb{D}$ to be the closed $2$-disk in the plane where one of the remaining moves is applied so that $K_{i-1} \cap (\mathbb{R}^2 \setminus \mathbb{D}) = K_{i} \cap (\mathbb{R}^2 \setminus \mathbb{D})$. We apply the braiding process to $K_{i-1} \cap (\mathbb{R}^2 \setminus \mathbb{D}) = K_{i} \cap(\mathbb{R}^2 \setminus \mathbb{D})$ to get diagrams $\widetilde{K}'_{i-1}$ and $\widetilde{K}'_{i}$ such that $\widetilde{K}'_{i-1} \cap \mathbb{D} = K_{i-1} \cap \mathbb{D}$, $\widetilde{K}'_{i} \cap \mathbb{D} = K_{i} \cap \mathbb{D}$ and $\widetilde{K}'_{i-1} \cap (\mathbb{R}^2 \setminus \mathbb{D}) =\widetilde{K}'_{i} \cap (\mathbb{R}^2 \setminus \mathbb{D})$.
\medskip

\noindent {{Case III.}} If $K_{i}$ is obtained from $K_{i-1}$ by an $R_{1a}$ or $R_{1b}$ move, then after the braiding process, the diagrams $\widetilde{K}'_{i-1}$ and $\widetilde{K}'_{i}$ looks like as in Figure \ref{R1Move}. Note that up to conjugation, virtual twins obtained from $\widetilde{K}'_{i-1}$ and $\widetilde{K}'_{i}$ are either of the following forms
$$\beta \text{ and } \beta \tau_n \tau_{n-1} \dots \tau_{i} \tau_{i-1} \tau_{i} \dots \tau_{n-1} \tau_{n}$$ or  $$\beta \text{ and } (1 \otimes \beta) \tau_1 \tau_{2} \dots \tau_{i-1} \tau_{i} \tau_{i-1} \dots \tau_{2} \tau_{1},$$ where $\beta \in VT_n$, $\tau_j=s_j$ or $\rho_j$ and $1 \leq i \leq n$. In each case, both the virtual twins are equivalent to each other by Lemma \ref{Lemma1} or Lemma \ref{Lemma2}. Thus, $\widetilde{K}_{i-1}$ and $\widetilde{K}_{i}$ are related by $M0-M5$ moves.
\begin{figure}[hbtp]
\centering
\includegraphics[scale=0.3]{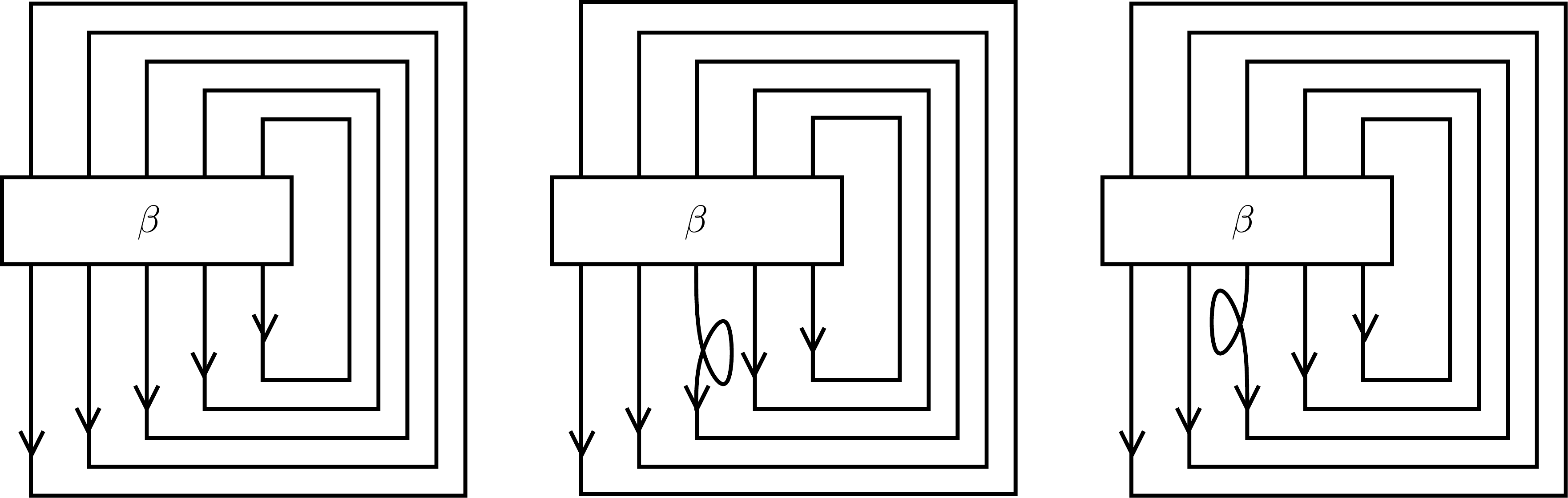}
\caption{$\widetilde{K}'_{i-1}$ and $\widetilde{K}'_i$ corresponding to $R_{1a}$ or $R_{1b}$ move}
\label{R1Move}
\end{figure}
\medskip

\noindent {{Case IV.}} If $K_{i}$ is obtained from $K_{i-1}$ by an $MVR_1$ move, then after braiding process, the diagrams $\widetilde{K}'_{i-1}$ and $\widetilde{K}'_{i}$ looks as in Figure \ref{MVR1Move}. The virtual twins obtained from $\widetilde{K}'_{i-1}$ and $\widetilde{K}'_{i}$ are of the form
$$\tau_n \tau_{n-1} \dots \tau_{i+1} s_i \beta_1 s_i \tau_{i+1} \dots \tau_{n-1} \tau_n \beta_2$$
and 
$$\tau_n \tau_{n-1} \dots \tau_{i+1} \rho_i \beta_1 \rho_i \tau_{i+1} \dots \tau_{n-1} \tau_n \beta_2,$$
respectively. By Lemma \ref{Lemma1}, both these virtual twins are equivalent, and hence $\widetilde{K}_{i-1}$ and $\widetilde{K}_{i}$ are related by $M0-M5$ moves.
\begin{figure}[hbtp]
\centering
\includegraphics[scale=0.35]{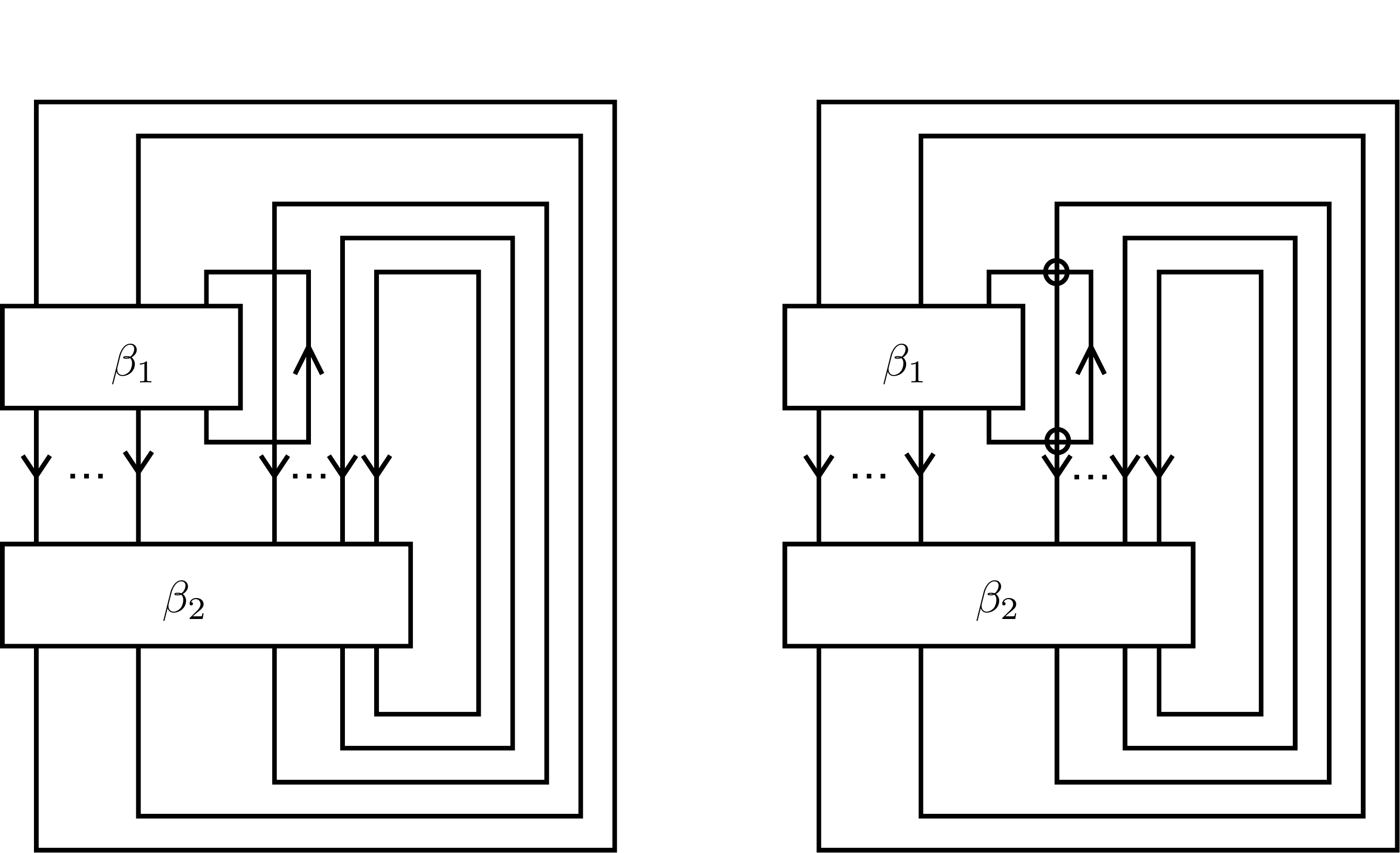}
\caption{$\widetilde{K}'_{i-1}$ and $\widetilde{K}'_i$ corresponding to $MVR_1$ move}
\label{MVR1Move}
\end{figure}
\begin{figure}[hbtp]
\centering
\includegraphics[scale=0.35]{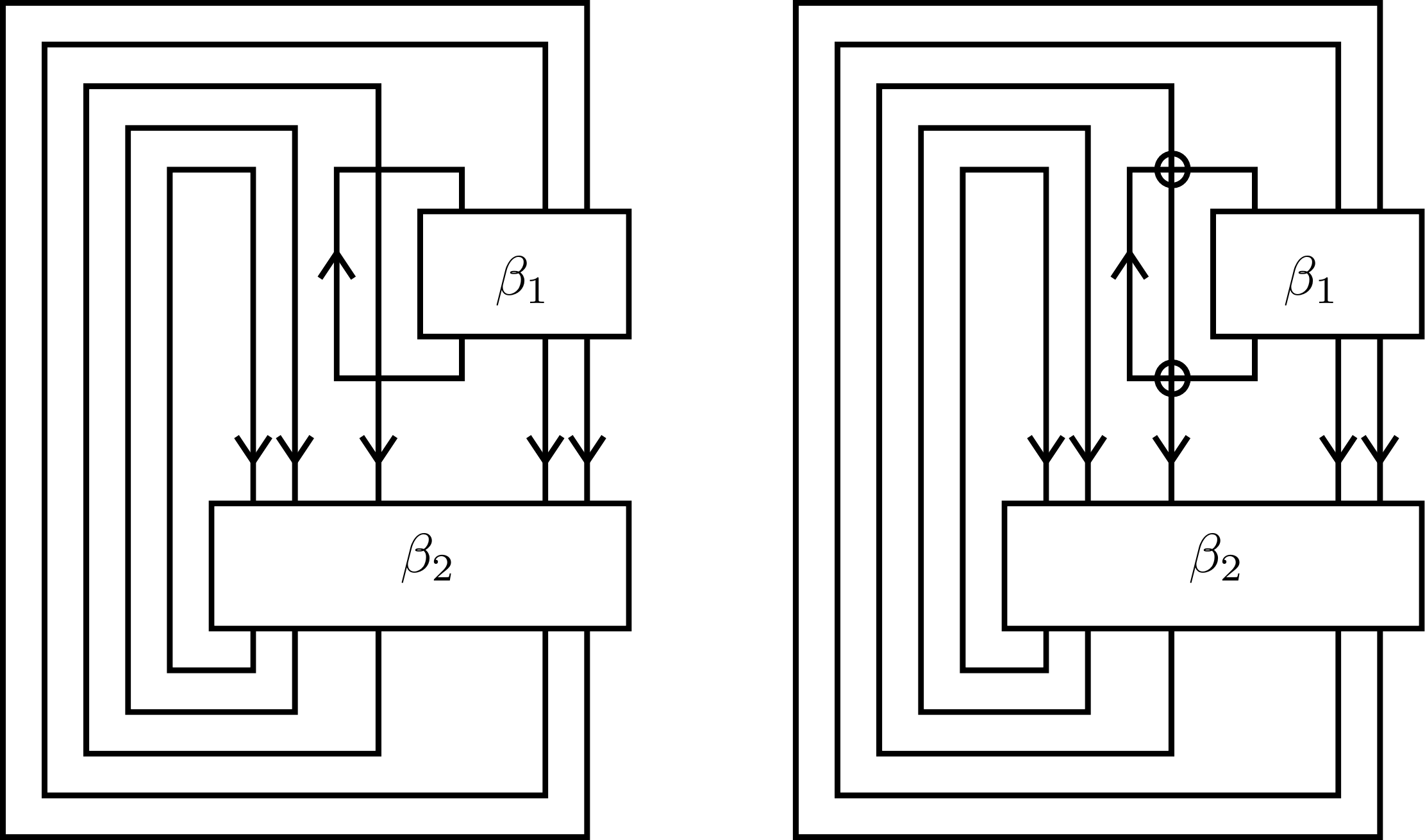}
\caption{$\widetilde{K}'_{i-1}$ and $\widetilde{K}'_i$ corresponding to $MVR_2$ move}
\label{MVR2Move}
\end{figure}
\medskip

\noindent {{Case V.}} If the move applied is $MVR_2$, then after the braiding process, the diagrams $\widetilde{K}'_{i-1}$ and $\widetilde{K}'_{i}$ looks as in Figure \ref{MVR2Move}. The virtual twins obtained from $\widetilde{K}'_{i-1}$ and $\widetilde{K}'_{i}$ are of the form
$$\tau_1 \tau_{2} \dots \tau_{i-1} s_i (i \otimes \beta_1) s_i \tau_{i-1} \dots \tau_{2} \tau_1  (1 \otimes \beta_2)$$
and
$$\tau_1 \tau_{2} \dots \tau_{i-1} \rho_i (i \otimes \beta_1) \rho_i \tau_{i-1} \dots \tau_{2} \tau_1  (1 \otimes \beta_2),$$
respectively. By Lemma \ref{Lemma2}, both of these virtual twins are equivalent, and hence $\widetilde{K}_{i-1}$ and $\widetilde{K}_{i}$ are related by $M0-M5$ moves.
\end{proof}
\medskip

\begin{ack}
The authors are grateful to the anonymous referees for their detailed reports which have substantially improved the readability of the paper. Neha Nanda thanks IISER Mohali for the PhD Research Fellowship. Mahender Singh is supported by the Swarna Jayanti Fellowship grants DST/SJF/MSA-02/2018-19 and SB/SJF/2019-20.
\end{ack}

\end{document}